\documentclass[12pt,oneside]{amsart}

\usepackage{amssymb}
\usepackage{amsmath}
\usepackage{graphicx}
\usepackage{subfigure}
\graphicspath{ {Images/} }
\usepackage{subfig}

\usepackage{float}
\usepackage{multicol}

      \theoremstyle{plain}
      \newtheorem{theorem}{Theorem}[section]
      \newtheorem{lemma}[theorem]{Lemma}
      \newtheorem{corollary}[theorem]{Corollary}
      
      \theoremstyle{definition}
      \newtheorem{definition}[theorem]{Definition}
      
      \theoremstyle{remark}
      \newtheorem{remark}[theorem]{Remark}
      
         \theoremstyle{problem}
      \newtheorem{problem}[theorem]{Problem}
      
         \theoremstyle{proposition}
      \newtheorem{proposition}[theorem]{Proposition}
      
     \theoremstyle{hypothesis}

      \makeatletter
      \def\@setcopyright{}
      \def\serieslogo@{}
      \makeatother
  \DeclareMathOperator{\cdim}{cdim}
 \DeclareMathOperator{\CHull}{CHull}
\begin{document}

\title[Representation of Convex Geometries]{Representation of Convex Geometries\\ by circles on a Plane}
\author {K. Adaricheva}
\address{Department of Mathematics, School of Science and Technology, Nazarbayev University,
53 Kabanbay Batyr ave., Astana, 010000 Republic of Kazakhstan}
\email{kira.adaricheva@nu.edu.kz}

\author {M. Bolat}
\address{School of Science and Technology, Nazarbayev University,
53 Kabanbay Batyr ave., Astana, 010000 Republic of Kazakhstan}
\email{madina.bolat@nu.edu.kz}

\thanks{}, 

\date{\today}

\begin{abstract}
 \noindent  Convex geometries are closure systems satisfying the anti-exchange axiom. Every finite convex geometry can be embedded into a convex geometry of finitely many points in an n-dimensional space equipped with a convex hull operator, by the result of  K.~Kashiwabara, M.~Nakamura and Y.~Okamoto (2005). Allowing circles rather than points, as was suggested by G.~Cz\'edli (2014), may presumably reduce the dimension for representation. This paper introduces a property, the Weak $2\times 3$-Carousel rule, which is satisfied by all convex geometries of circles on a plane, and we show that it does not hold in all finite convex geometries. This raises a number of representation problems for convex geometries, which may allow us to better understand the properties of Euclidean space related to its dimension.
\end{abstract}

\maketitle

   \section{Introduction}
Convex geometries were studied from different perspectives and under different names since the 1930s. R.P.~Dilworth \cite{D2} knew them as lattices with unique irredundant decompositions, and B.~Monjardet mentions many ways the convex geometries were rediscovered before the mid-80s \cite{Mo85}. An important survey by P.H.~Edelman and R.E.~Jamison \cite{EdelJam} included results on several equivalent definitions of finite convex geometries and outlined a program for future studies, suggesting a list of open problems. 

Convex geometries are interesting combinatorial objects which generalize a notion of convexity in Euclidean space. There are many structures which share their properties. Among them are convex objects in Euclidean space, convex sets in posets, subsemilattices in a semilattice, and others that are considered in \cite{EdelJam}.  But the main driving example is a geometrical one: a set of points in Euclidean space equipped with a closure operator of the convex hull. This convex geometry nowadays is called \emph{affine} convex geometry. Among  important problems raised in \cite{EdelJam} was the following:

     \begin{problem}\label{prob1}
  Describe affine convex geometries, i.e., list the properties which ensure that a finite convex geometry can be represented by convex subsets of some finite point configuration in a finitely dimensional Euclidean space. 
   \end{problem}

A few partial results were obtained towards the solution of this problem, showing that the problem is rather difficult.
First, affine convex geometries, represented by simple configuration on a plane: a single point inside a $n$-gon - were described by P.H.~Edelman and D.G.~Larman \cite{EdLa}. Already the simplest generalization, two points inside a $n$-gon on a plane, became a much harder problem, as shown in K. Adaricheva \cite{Adar2007}. More recently, K.~Adaricheva and M.~Wild \cite{AdarWild} have shown that some modification of Problem \ref{prob1} is (polynomially) equivalent to the Order Type problem, which is NP-hard, the latter result following from the Universality Theorem of N.E. Mn\"ev \cite{Mn}.    

A new stage of studies of convex geometries started from the publication of K.~Adaricheva, V.~Gorbunov and V.~Tumanov \cite{AdarGT}, when important generalizations were made from the finite case to the infinite one.
In that paper some important results were proved with respect to the \emph{embedding} of finite geometries, treated as finite lattices, into larger lattices, which would inherit some properties one would like to see in the infinite version of convex geometries. In particular, one of the key results of the paper was the finding a \emph{universal class} for finite convex geometries, i.e. a specific class of convex geometries, which would accept any finite convex geometry as a \emph{sub-geometry}. That specific class comprises lattices of algebraic subsets of algebraic lattices, which are not necessarily finite. For that reason, the paper has also raised the problem of finding a universal class comprising only \emph{finite} convex geometries:
 
   \begin{problem}\label{prob2}
  Find a class $\mathcal{C}$ of \emph{finite} convex geometries, such that every finite convex geometry is a sub-geometry of some geometry in $\mathcal{C}$.
   \end{problem}

One natural candidate for such representation was also indicated in \cite{AdarGT}: the class of affine convex geometries. Two similar results were proved, by different methods, in F.~Wehrung and M.~Semenova \cite{WS04} and K.~Adaricheva \cite{Adar2004}, showing that the class of affine convex geometries is universal for finite convex geometries without cycles, a.k.a convex geometries which are \emph{lower bounded}, if they are  treated as lattices. Lower bounded lattices play an important role in lattice theory, especially in the study of free lattices, see monograph R.~Freese, J.~ Je\v{z}ek and J.B.~Nation \cite{FJN}. Finally, it was shown by K.~Kashiwabara, M.~Nakamura and Y.~Okamoto \cite{Kashetal}  that every finite convex geometry is a sub-geometry of some affine convex geometry in $\mathbb{R}^n$, solving Problem \ref{prob2} in the positive. 

The two problems above showcase the different types of representation for convex geometries, and the second one is a generalization of the first one. While the first one assumes the \emph{isomorphism} between lattices of closed sets of convex geometries, the second requires \emph{embedding}. We note that embedding constitutes the classical algebraic approach to representation problems, allowing a lot of flexibility. To distinguish between the two types of representations for convex geometries, we will call the representation by isomorphism a \emph{strong representation}, and by embedding a \emph{weak representation}.

To summarize, we say that the weak representation problem for finite convex geometries is now solved, with respect to at least two universal classes: finite affine convex geometries and lattices of algebraic subsets of algebraic lattices. With respect to the first universal class, the problem of finding the optimal dimension for representation clearly arises: what is the smallest dimension $n$ of space $\mathbb{R}^n$ such that a given convex geometry is a sub-geometry of affine geometry in $\mathbb{R}^n$? For example, construction in \cite{Kashetal} would require $n=O(|G|)$, where $G$ is a ground set of convex geometry. On the other hand, if geometry is without cycles, the construction in \cite{Adar2004} would require a smaller dimension: $n=O(log(|G|)$.

Recently, M.~Richter and L.~Rogers \cite{RichRog} established an upper bound for the dimension $n$ of the weak representation of convex geometry by an affine geometry: $n \leq  \min(|G|,\cdim(G))$. Here $\cdim{G}$ is a \emph{convex dimension} of a geometry, a parameter discussed in \cite{EdelJam}.

Generally, it would be interesting to characterize convex geometries that can be weakly represented by affine geometries in $\mathbb{R}^n$. The necessary condition for such representation, called the $n$-Carousel rule, was established in K.~Adaricheva \cite{Adaricheva}. It was also shown that it implies the $n$-Carath\'eodory property.

As far as the strong representation is concerned, a new idea was introduced by G.~Cz\'edli \cite{Czedli}, who suggested using circles rather than points in $\mathbb{R}^2$ for the (strong) representation of convex geometries of the convex dimension $2$. Apparently, circles provide much more flexibility for the strong representation of convex geometries. Indeed, affine convex geometries are \emph{atomistic}, i.e., one-element subsets of the ground set of convex geometry are always convex. This would restrict the possibility of the strong representation by affine geometries to \emph{atomistic} convex geometries only. Replacing points by circles (or, more commonly, by balls in $\mathbb{R}^n$) removes such restriction for the strong representation.  

Taking the idea even further, M.~Richter and L.~Rogers in \cite{RichRog} showed how to use \emph{polygons} for the strong representation of convex geometries on the plane. In a sense, this is a nice visualization of the Theorem in \cite{EdelJam} about compatible orderings of a convex geometry and its parameter $\cdim$. We note though that polygons that appear in this result are not necessarily convex themselves. Also, the relative position of such polygons is very restricted, unlike the case with representation by circles.

In our paper we tackled the problem of representation by circles on a plane, discussed first in \cite{CzedliComm}. It was natural to venture, after the result of \cite{Czedli}, whether arbitrary convex geometries, in particular, those with $\cdim > 2$, could be strongly represented by circles on a plane. 

In this paper, we answer this question \emph{in negative}, demonstrating an example of convex geometry with $\cdim = 6$, which cannot  have such a representation.

Moreover, we prove in Theorem \ref{theorem2} that all convex geometries having strong representation by circles on a plane satisfy the property which we call the \emph{Weak $2\times 3$-Carousel rule}. Indeed, this property is a weakening of the $2$-Carousel rule, first introduced in \cite{AdarWild}. 

Naturally, this brings a question generalizing Problem \ref{prob1}, which was first formulated in \cite{CzedliComm}:

\begin{problem}
Is every finite convex geometry (strongly) represented by a convex geometry of balls in $\mathbb{R}^n$? 
\end{problem} 

The paper is organized as follows: In section \ref{Definitions} we introduce the main concepts and definitions, and provide several statements in the background of our investigation. 

 In Section \ref{WeakCarouselTriangles} we give the outline of the proof of the main Theorem and prove the Weak Carousel property for triangles, a partial case of the Weak $2\times 3$-Carousel rule in the geometry of circles on a plane. The proof of the Weak Carousel property for triangles requires a number of geometrical results which are stated and proved in Section \ref{Lemmas}. 

The appendices keep the illustrations of realizable cases of the location of two circles inside a triangle. 
The main result, the Weak $2\times 3$-Carousel rule for circles, is proved in Section \ref{sectionWCP}. Examples of affine convex geometry that fail the Weak $2\times 3$-Carousel rule is described in Section \ref{Example}.
We discuss the results and propose the problems for future study in Section \ref{CR}.
     
\section {Background} \label{Definitions}

In this section we introduce all necessary definitions and formulate basic statements.
The terminology follows \cite{EdelJam}.

\begin{definition}
Given any set $X$, a closure operator on $X$ is a mapping $\varphi: 2^{X}\rightarrow2^X$ with the following properties:
\begin{enumerate}
\item $Y\subseteq\varphi (Y)$ for every $Y\subseteq X$;
\item If $Y\subseteq Z$, then $\varphi(Y)\subseteq \varphi (Z)$ for $Y,Z \subseteq X$;
\item $\varphi(\varphi(Y))=\varphi(Y)$ for $Y\subseteq X$.
   \end{enumerate}
   \end{definition}
 
Set $X$ will be called a \emph{ground set}, or a \emph{base set} for a closure system, the latter being defined as a pair $(X, \varphi)$. We can also associate the closure system with a special family of subsets called an alignment.

\begin{definition}\label{Alignment}
Given any (finite) set $X$, an \textit{alignment} on $X$ is a family $\mathcal{F}$ of subsets of $X$ which satisfies two properties:
\begin{enumerate}
\item $X\in \mathcal{F}$;
\item If $Y, Z \in \mathcal{F}$, then $Y \cap Z \in \mathcal{F}$.
  \end{enumerate}
  \end{definition}
 Note that the definition of alignment requires a slight modification, when the ground set is not assumed to be finite. We will consider only finite ground sets within the current paper.
 
It is well-known that any alignment on a finite ground set forms a lattice, where the meet operation $\wedge$ is the set intersection $\cap$, and the join operation $\vee$ is defined as follows: $Y\vee Z = \cap \{W \in \mathcal{F}: Y, Z \subseteq W\}$.

The following relationships between a closure operator and an alignment could be easily verified.
  
  \begin{proposition}
Let $X$ be some finite ground set.
\begin{enumerate}
\item If $\varphi$ is a closure operator on $X$, then $\mathcal{F}=\{ Y\subseteq X:\varphi(Y)=Y\}$ is an alignment on $X$. 
\item Let $\mathcal{F}$ be an alignment on $X$. Define $\varphi (Y)=\cap\{Z\in \mathcal{F}: Y\subseteq Z\}$ for every $Y\subseteq X$. Then, $\varphi$ is a closure operator on $X$.
\item The correspondences between a closure operator and an alignment on $X$ in items (1) and (2) are inverses of each other. 
\end{enumerate}
 \end{proposition}

\emph{An implication} $Y\rightarrow Z$ of a closure system $(X,\varphi)$ is a statement that $Z\subseteq \phi(Y)$. A set of such implications is called \emph{an implicational basis}, if any implication that holds in 
$(X,\varphi)$ is a logical consequence of the basis. The study of implicational bases of closure systems, and convex geometries in particular is quite active, see K. Adaricheva and J.B. Nation \cite{AN16} and M. Wild \cite{Wi16}.

We turn now to special properties of a closure operator or alignment, which distinguish convex geometries.
 
 \begin{definition}
Closure system $(X,\varphi)$ is called a \emph{convex geometry} if $\varphi$ is a closure operator on $X$ with additional properties:
\begin{enumerate}
\item $\varphi(\emptyset)=\emptyset$;
\item if $Y=\varphi(Y)$ and $x, z\notin Y$, then $z\in\varphi(Y\cup x)$ implies that $x\notin\varphi(Y\cup z)$ (Anti-exchange property).
\end{enumerate}
\end{definition}

Convex geometries could be defined equivalently through an alignment.

    \begin{definition}
 Pair $(X,\mathcal{F})$ is called a \textit{convex geometry} if $\mathcal{F}$ is an alignment on $X$ with additional properties:
 \begin{enumerate}
\item $\emptyset\in \mathcal{F}$;
\item if $Y\in \mathcal{F}$ and $Y \neq X$, then $\exists a\in X\setminus Y$ s.t. $Y\cup\{a\}\in \mathcal{F}$. 
  \end{enumerate}
  \end{definition}

Theorem 2.1 in \cite{EdelJam} establishes, among other statements, that two above definitions of convex geometry are equivalent.

There is a simple type of alignments whose elements form a chain.

\begin{definition}
An alignment $\mathcal{F}$ defined on a set $X$ is called a \emph{monotone} alignment if there is a total ordering $x_1<x_2<\dots<x_n$ such that $\{x_1, x_2, \dots, x_i\}\in\mathcal{F}$ for all $i, 1\leq i \leq n$, and these sets are the only elements of $\mathcal{F}$.
\end{definition}
 It is straightforward to verify that pair  $(X,\mathcal{F})$ is always a convex geometry, for any monotone alignment $\mathcal{F}$ on $X$.

Given two alignments $\mathcal{F}_1$ and $\mathcal{F}_2$ defined on the same base set $X$, an operation of join on these alignments is defined as follows 
:
\[
\mathcal{F}_1 \vee \mathcal{F}_2=\{S\subseteq X: S=U\cup V \text{ for } U\in \mathcal{F}_1 \text{ and } V\in \mathcal{F}_2\}
\]

The following result was proved as Theorem 5.1 in \cite{EdelJam}.

\begin{theorem}
If $(X,\mathcal{F}_1)$ and $(X,\mathcal{F}_2)$ are convex geometries, then $(X,\mathcal{F}_1\vee \mathcal{F}_2)$ is a convex geometry.  
\end{theorem}

It turns out that an alignment of any convex geometry could be viewed as a join of several monotone alignments \cite[Theorem 5.2]{EdelJam}:

\begin{theorem}
Given any convex geometry $G=(X,\mathcal{F})$, $\mathcal{F}=\bigvee_{i \leq n} \mathcal{L}_i$, for some $n \in \mathbb{N}$, where $\mathcal{L}_i$ is a monotone alignment defined on $X$, for every $i\leq n$.
\end{theorem}

As a consequence, it is of interest to define a parameter associated with convex geometry to represent a minimal number of monotone alignments needed to realize an alignment of convex geometry:

\begin{definition}\cite{EdelJam}
Given convex geometry $G=(X,\mathcal{F})$, convex dimension $\cdim$ of G  is a minimal number of monotone alignments needed to realize $\mathcal{F}$.
\end{definition}

The following example remains to be the main driving model of convex geometries.

\begin{definition}
\textit{An affine convex geometry} is a convex geometry $C_0(\mathbb{R}^n, X)=(X, ch)$, where $X$ is a set of points in $\mathbb{R}^n$ and $ch$ is closure operator of relative convex hull, which is defined as follows:
for $Y\subseteq X, ch(Y)=\CHull(Y)\cap X$, where $\CHull$ is a usual convex hull operator. 
\end{definition}

Affine convex geometries form a sub-class of atomistic closure systems.

\begin{definition}
A closure system $G=(X, \varphi)$ is called \textit{atomistic} if $\varphi(\{x\})=\{x\}$ for every $x\in X$. 
\end{definition}

Indeed, for any affine convex geometry $G=(X, ch)$, any $x\in X$ is a point in $\mathbb{R}^n$ and $ch(\{x\})=\{x\}$. 

The following generalization from points to circles was suggested in \cite{Czedli}. It can easily be generalized to the balls in $\mathbb{R}^n$, but we do not need this generalization in the current paper. Note that points are also considered as circles, whose radii are zero.

\begin{definition}\label{GeomCircles}
Consider closure system $F=(X, ch_{c})$, where $X$ is a set of circles in $\mathbb{R}^2$ and closure operator $ch_{c}$ is defined as follows:  $ch_{c}(Y)=\{z\in X: \tilde{z}\subseteq \CHull(\cup \tilde{y}: y\in Y)\}$ for $Y\subseteq X$, where $\CHull$ is a usual convex hull operator and $\tilde{x}$ is a set of points in $x \in X$. We call $F$ a \textit{geometry of circles on a plane}. 
\end{definition}

\begin{proposition}\cite{Czedli} For any finite set of circles $X$ on a plane, the closure system $F=(X, ch_{c})$ is a convex geometry.
\end{proposition}

A geometry of circles on a plane is not atomistic in general. For convex geometry of circles $F=(X, ch_c)$, it is possible that $ch_c(\{x\})=\{x,y\}$ for $x, y \in X$ that describes a case when circle $y$ is inside circle $x$. 

Important concept for the current paper relates two geometries through the mapping of embedding.
    
     \begin{definition}\label{Subgeometry}
Convex geometry $G_1=(X,\mathcal{F}_1)$ is a sub-geometry of $G_2=(Y, \mathcal{F}_2)$ if there is a one-to-one map $f:\mathcal{F}_1\rightarrow \mathcal{F}_2$ s.t. 
\begin{enumerate}
\item $f(A\cap B)=f(A)\cap f(B)$, $A, B \subseteq X$;
\item $f(A\vee B)=f(A)\vee f(B)$, $A, B \subseteq X$.
 \end{enumerate}
 \end{definition}

Another way to connect these two geometries is to say that geometry $G_1$ has a \emph{weak representation} in $G_2$.

If map $f$ is also onto, then we talk about \emph{strong representation} of $G_1$ in $G_2$, or isomorphism of convex geometries. Such mapping also induces bijection $f_g$ between ground sets $X$ and $Y$, namely, $f_g(x)=y$ iff $f(\phi_1(\{x\})) = \phi_2(\{y\})$, where $\phi_1$ and $\phi_2$ are closure operators corresponding to alignments $\mathcal{F}_1$ and $\mathcal{F}_1$, respectively.

The following definition was introduced in \cite{AdarWild} and played a crucial role in \cite{Adaricheva}. It was shown in the latter paper that it implies the well-known \emph{$n$-Carath\'eodory property}, which describes essential behavior of the convex hull operator in the $n$-dimensional Euclidean space: if a point $x$ is in the convex hull of set of points $S$, then it is in the convex hull of some $n+1$ points from $S$.

   \begin{definition}\label{Carousel}
A convex geometry $(A, \varphi)$ satisfies the \textit{n-Carousel rule} if $x, y \in \varphi(S), \\S\subseteq A$, implies $x\in\varphi(\{y, a_1,\dots,a_n\})$ for some $a_1,\dots, a_n\in S$.
  \end{definition}
 
Note that $x$ and $y$ play symmetric roles in the statement, therefore, we simultaneously obtain that 
 $y\in\varphi(\{x, b_1,\dots, b_n\})$ for some $b_1,\dots, b_n\in S$.

In this paper we consider a slight variation of this property, replacing logical conjunction by the disjunction in the conclusion of the statement.
  
   \begin{definition}\label{CGWeakCarousel}
A convex geometry $(A, \varphi)$ satisfies the \textit{Weak $n$-Carousel rule} if $x, y \in \varphi(S), S\subseteq A$, implies either $x\in\varphi(\{y, a_1,\dots, a_n\})$ or $y\in\varphi(\{x, a_1, \dots, a_n\})$ for some $a_1,\dots, a_n\in S$.
  \end{definition}
   
If we set $n=2$ and $|S|=3$, then we can formulate a partial case of the Weak 2-Carousel rule as follows.  

    \begin{definition}\label{WeakCarousel}
A convex geometry $(A, \varphi)$ satisfies the \textit{Weak $2\times 3$-Carousel rule} if $x, y \in \varphi(S), S\subseteq A$ and $|S|=3$, implies either $x\in\varphi(\{y, a_1,a_2\})$ or $y\in\varphi(\{x, a_1,a_2\})$ for some $a_1,a_2\in S$.    

   \end{definition}
  
  To prove Weak $2\times 3$-Carousel rule for circles, we will first consider special case of the property, when set $S$ consists of points, a.k.a. circles of radii zero. This will result in the following target geometrical statement  in section \ref{WeakCarouselTriangles}.
  
\begin{definition}\label{WeakCarouselTr}
    A configuration of two circles $x$ and $y$ and a set $S$ of distinct points $A, B, C$ on a plane, not on one line, is said to satisfy the \textit{Weak Carousel Property for Triangles}, if 
$x,y \in ch_c(\{A,B,C\})$ implies $x\in ch_c(y, U, V)$ or $y\in ch_c(x, U, V)$, for some $U,V\in \{A,B,C\}$.
   \end{definition}

\section {The Weak Carousel Property for Triangles}\label{WeakCarouselTriangles}

In this section we prove that convex geometry of circles on a plane satisfies 
the Weak Carousel property for Triangles, which is a partial case of the Weak $2\times 3$-Carousel rule. 
In section \ref{sectionWCP} we will generalize it to show that the Weak $2\times 3$-Carousel rule also holds in this geometry.

We consider two circles inside a triangle and model their location by their projections on sides of this triangle. Thus, we transform a problem from considering positions of circles to looking at configurations of segments on the sides of triangle.

We find all possible configurations of segments which we reduce to smaller number of cases up to isomorphism. We dismiss some of these cases by proving that a location of circles with any of these projections is not realizable. For this, we prove a number of lemmas, which will be placed in section \ref{Lemmas}. We show that all other cases are realizable, and the Weak Carousel property for Triangles holds there. 

 \begin{theorem}\label{theorem1}
   Every configuration of two circles and a set of three distinct points in $\mathbb{R}^2$, not on one line, satisfies the Weak Carousel property for Triangles. 
   \end{theorem}
Note: we will call the Weak Carousel property for Triangles within this and next sections simply the Weak Carousel property, to shorten the notation. 
\begin{proof}[Proof]
Consider two circles $x$ and $y$ in a triangle $\triangle{ABC}$. Circles $x$ and $y$ are projected to segments on sides of $\triangle{ABC}$. As in the example in Figure 1, we will be using notation, say, $x_B^{BC}$ and $x_C^{BC}$, for edges of a projection of a circle $x$ from vertex $A$ on $BC$ which are closest to $B$ and $C$, respectively.

\begin{figure}[h!]
\includegraphics[width=0.3\textwidth]{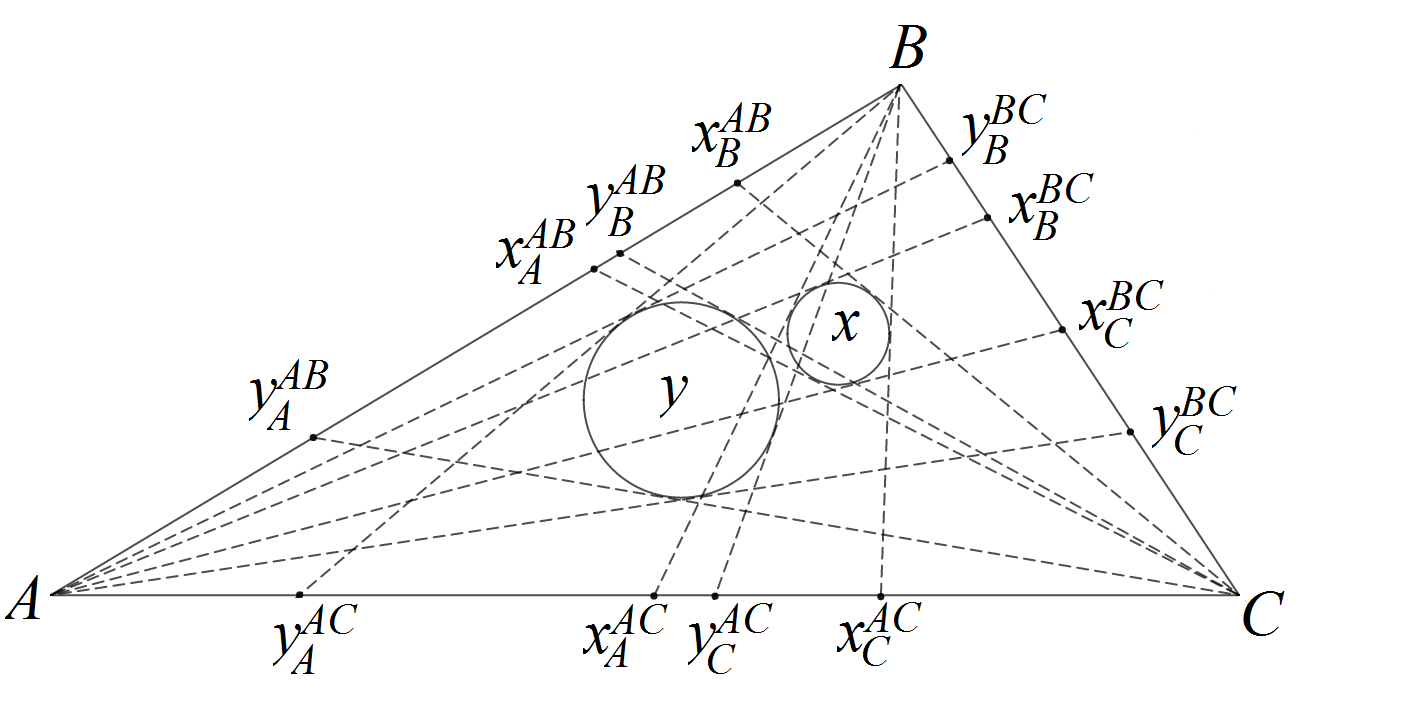}  
\caption{ }
  \label{fig:Circles}
\end{figure}
There are six possible configurations of projections of $x$ and $y$ on each side of a triangle, when \emph{all four end points on each side are distinct}. We will consider configurations which allow coincident end points later. 

We assign a number $i$, that takes integer values from $1$ to $6$, to each configuration of segments on one side of $\triangle{ABC}$. We illustrate configurations corresponding to values of $i$ in Figure \ref{fig:Segments} in Appendix \ref{App:AppendixA}. We assume that $x_1$ and $y_1$ in Figure \ref{fig:Segments} are first in orders $x_1x_2$ and in $y_1y_2$ respectively in a clockwise walk around $\triangle{ABC}$, assuming that $ABC$ is in a clockwise order. 

We denote a configuration of projections of two circles inside $\triangle{ABC}$ by $C_{jkl}$, where $j, k, l$ are integer numbers that take values of $i$ and denote a configuration of segments on sides $AB, BC, CA$, respectively. Say, example from Figure \ref{fig:Circles} is $C_{546}$. 

Six possible configurations for each side produce 216 possible configurations of projections of $x$ and $y$ for  $\triangle{ABC}$. These configurations could be grouped into 38 classes up to isomorphism. First, $C_{jkl}$, $C_{klj}$ and $C_{ljk}$ are isomorphic since they represent the same triangle with sides rotated clockwise.  Second, configurations for a side of $\triangle{ABC}$ assigned by $i=1$ and $i=2$ are isomorphic since they represent the same configurations with $x$ and $y$ swapped. Similarly, for $i=3$ and $4$, $i=5$ and $6$ (Figure \ref{fig:Segments}). Say, $C_{241}$ is isomorphic to $C_{124}$ and $C_{412}$ due to the first reason, and $C_{241}$ is isomorphic to $C_{132}$ due to the second reason. We denote classes by $S_n$, where $n$ takes values from $1$ to $38$. Say, if $C_{241}$ is in class $S_4$, then we include 1) $C_{124}$, $C_{412}$; 2) $C_{132}$; and 3) $C_{321}$, $C_{213}$, which are isomorphic cases of $C_{132}$. Refer to Appendix \ref{App:AppendixB} for a full list of grouping 216 configurations in 38 classes.

We start from showing that some of the classes of configurations cannot be realized by placing any two circles into a triangle.

\begin{enumerate}
\item Consider classes $S_2$, $S_9,S_{25},S_{38}$, where projections of circles $x$ and $y$ appear in the same order, on each side of triangle, when moving around $\triangle{ABC}$ clockwise. Say, projections of $y$ may appear ``strictly later'' than projections of $x$: when moving from $A$ to $B$ one would meet $x_A^{AB}$ before $y_A^{AB}$, and $x_B^{AB}$ before $y_B^{AB}$, similarly, when moving from $B$ to $C$ and from $C$ to $A$. 

Figure \ref{fig:DismissCase1} illustrates a configuration $C_{222}$ as an example from class $S_2$. Take tangent lines to $x$, connecting the first edge point of each $x$-projection, which we meet when we walk around $\triangle{ABC}$ clockwise, with the opposite vertex of $\triangle{ABC}$ (See Figure \ref{fig:Circles1}). Then, circle $x$ is inscribed into $\triangle{A'B'C'}$ formed by these three lines.  It follows from assumption that $y$ is inside $\triangle{A'B'C'}$. Moreover, $y$ should have points in each of three disjoint areas of $\triangle{A'B'C'}$, whose union is $\triangle{ABC}\backslash{x}$. Then due to Corollary \ref{corollary4}, the case is dismissed as impossible for realization.

Configurations from classes $S_9,S_{25},S_{38}$ are dismissed using similar argument. 

\begin{multicols}{2}
   \begin{figure}[H]
\includegraphics[width=0.3\textwidth]{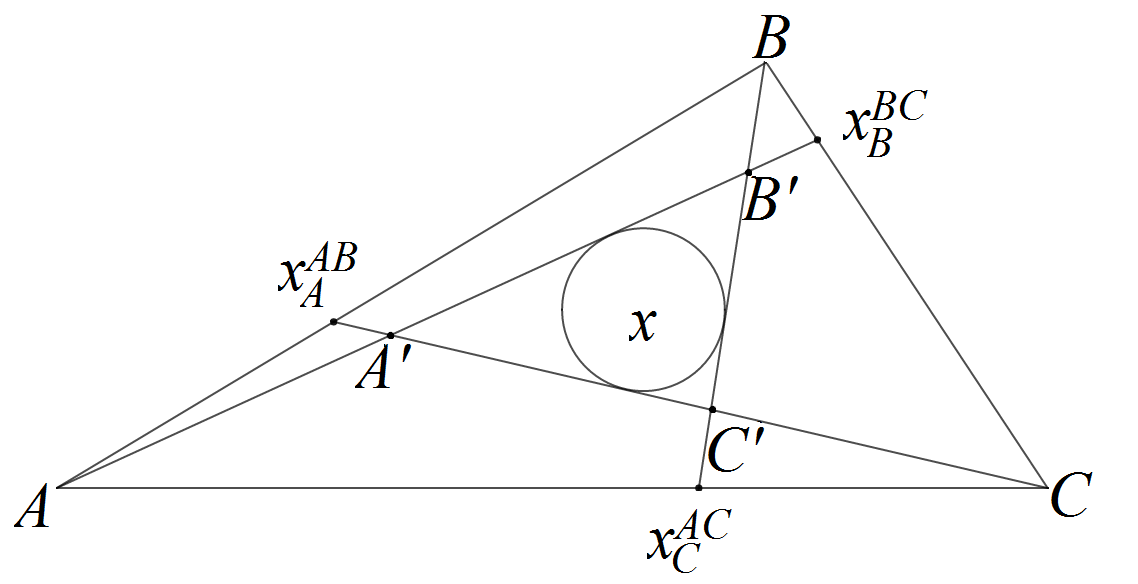}  
\caption{ }
  \label{fig:Circles1}
\end{figure}
\columnbreak
   \begin{figure}[H]
\includegraphics[width=0.3\textwidth]{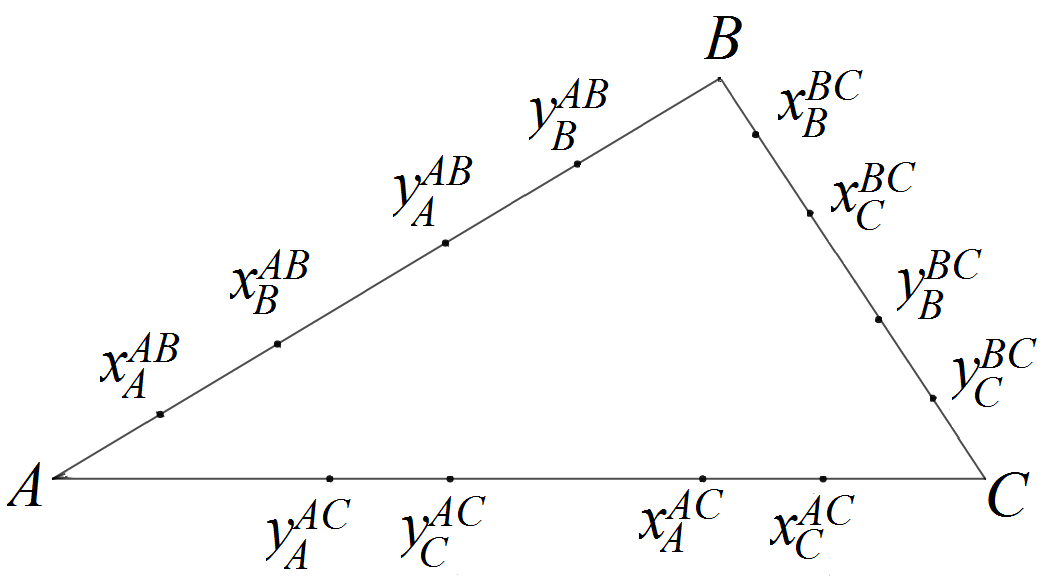}  
\caption{ }
  \label{fig:DismissCase1}
\end{figure}
\end{multicols}

\item Consider class $S_{13}$,where a segment of a projection of $y$ is
   inside a segment of a projection of $x$ on one side of
   $\triangle{ABC}$, but a projection of $x$ is inside a projection of
   $y$ on another side of $\triangle{ABC}$. Say, $y_B^{BC}y_C^{BC}$ is
   inside $x_B^{BC}x_C^{BC}$ and $x_A^{AC}x_C^{AC}$ is inside
   $y_A^{AC}y_C^{AC}$. Figure \ref{fig:DismissCase2} illustrates a
   configuration $C_{234}$ as an example from class $S_{13}$. Take
   tangent lines to $x$, connecting A with $x_B^{BC}$ and $x_C^{BC}$,
   see Figure \ref{fig:Circles2}. It follows from assumption that $y$
   is inside $\triangle{Ax_B^{BC}x_C^{BC}}$. Moreover, $y$ should have
   points in each of two disjoint areas of
   $\triangle{Ax_B^{BC}x_C^{BC}}$ whose union is $\triangle {Ax_B^{BC}x_C^{BC}}\backslash{x}$. Then due to Corollary \ref{corollary4}, the case is dismissed as impossible for realization.
   Configurations from classes $S_5,S_{12},S_{28},S_{31},S_{32}$ are dismissed using similar argument.   

\begin{multicols}{2}
 \begin{figure}[H]
\includegraphics[width=0.3\textwidth]{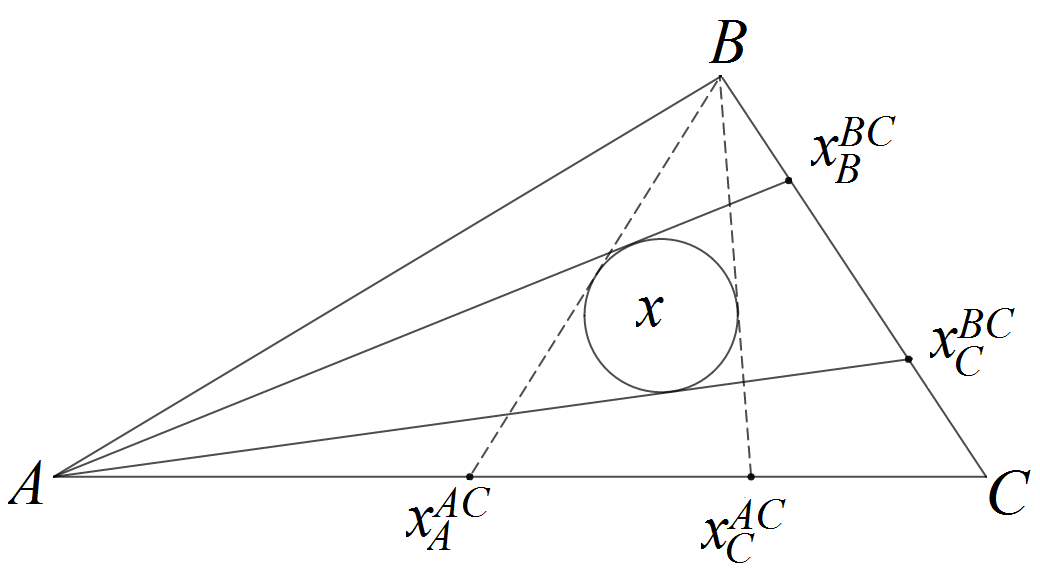}  
\caption{ }
  \label{fig:Circles2}
\end{figure}
\columnbreak
   \begin{figure}[H]
\includegraphics[width=0.3\textwidth]{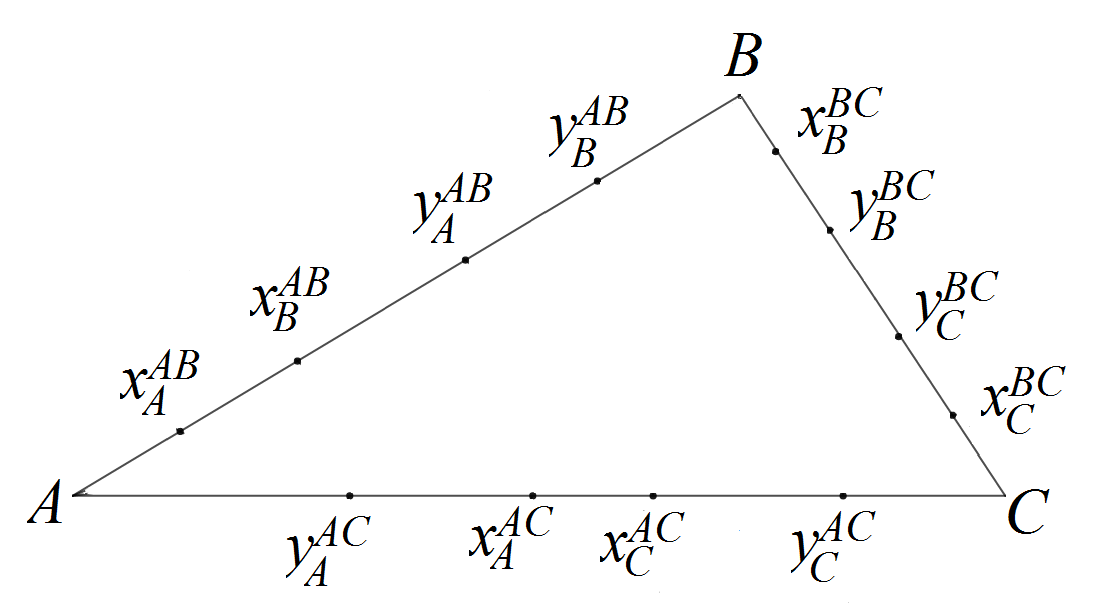}  
\caption{ }
  \label{fig:DismissCase2}
\end{figure}
\end{multicols}
 
\item Consider class $S_{15}$. Figure \ref{fig:DismissCase3} illustrates
   configuration $C_{135}$ as an example from class $S_{15}$. Take
   tangent lines to $x$, connecting A with $x_B^{BC}$ and $x_C^{BC}$,
   see Figure \ref{fig:Circles4}. It follows from assumption that $y$
   should have points in $\triangle{x_C^{AC}BC}$ and
   $\triangle{x_A^{AB}AC}$. Moreover, $y$ is inside
   $\triangle{Ax_B^{BC}x_C^{BC}}$. So, $y$ should have points in each of
   two disjoint areas of $\triangle{Ax_B^{BC}x_C^{BC}}$, whose union is
   $\triangle{Ax_B^{BC}x_C^{BC}}\backslash{x}$. Then due to Corollary \ref{corollary4}, the case is dismissed as impossible for realization.
   Configurations from classes $S_{17}, S_{21}, S_{22}, S_{34},S_{35}$ are dismissed using similar argument.   

\begin{multicols}{2}
 \begin{figure}[H]
\includegraphics[width=0.3\textwidth]{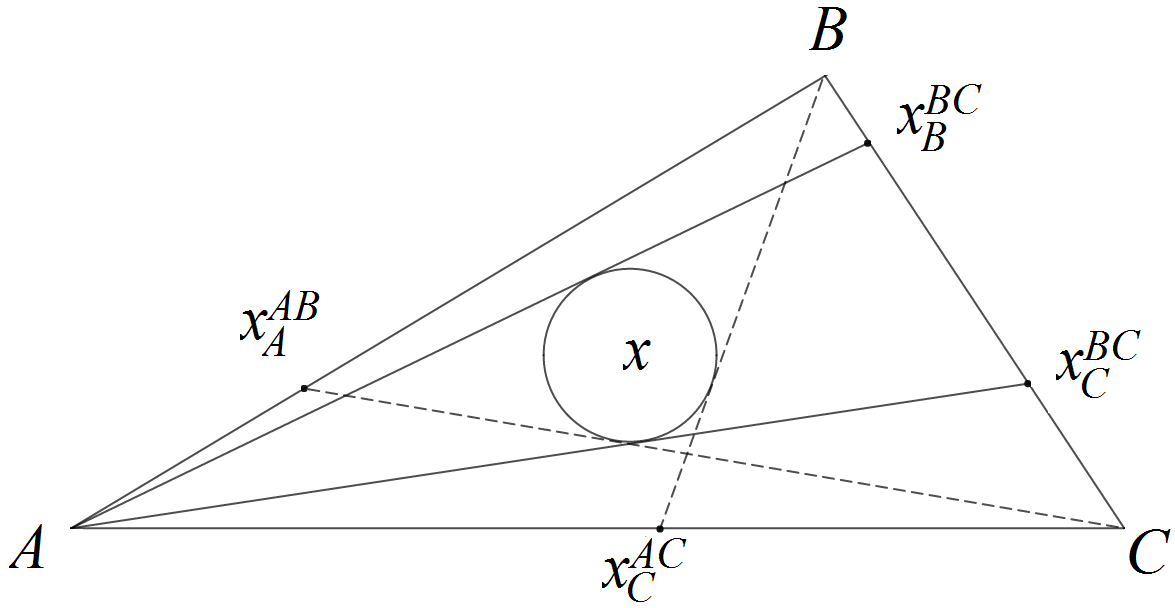}  
\caption{ }
  \label{fig:Circles4}
\end{figure}

\columnbreak

   \begin{figure}[H]
\includegraphics[width=0.3\textwidth]{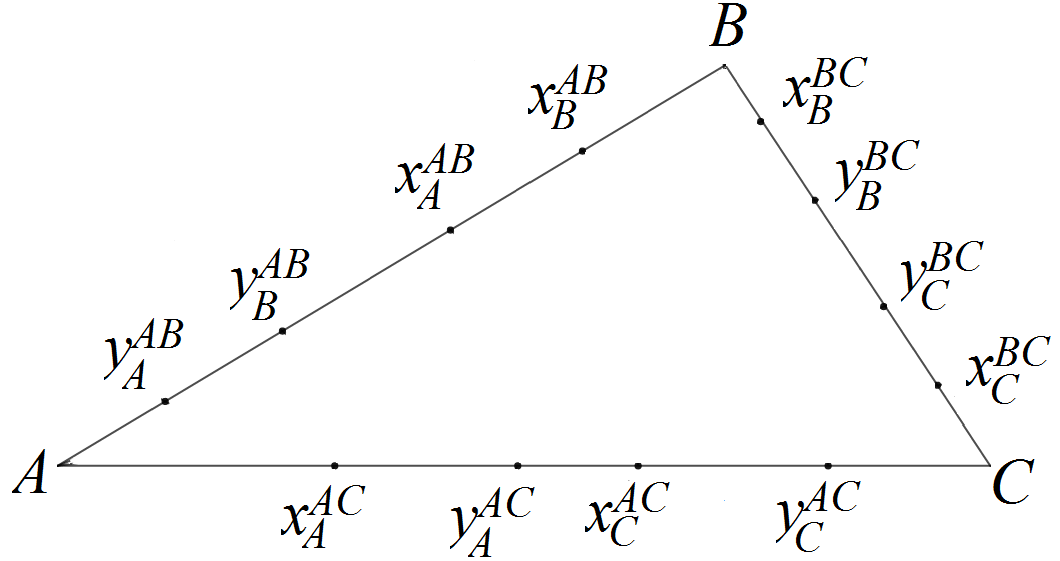}  
\caption{ }
  \label{fig:DismissCase3}
\end{figure}
 \end{multicols}

\item Consider class $S_{10}$. Figure \ref{fig:DismissCase4} illustrates
   configuration $C_{225}$ as an example from class $S_{10}$. Take
   tangent lines to $x$, connecting vertex $B$ with $x_A^{AC}$ and $x_A^{AC}$, vertex $A$
   with $x_C^{BC}$, and vertex $C$ with $x_B^{AB}$, see Figure \ref{fig:Circles3}.
   Let $Ax_C^{BC}$ and $Cx_B^{AB}$ intersect in point $O$. It follows
   from assumption that $y$ should be in $\triangle{COx_C^{BC}}$. 
   Moreover, $y$ should have a point in $\triangle {x_A^{AC}Bx_C^{AC}}$.
   Due to Lemma \ref{lemma6}, $O$ is not in $\triangle {x_A^{AC}Bx_C^{AC}}$, therefore, the case is dismissed as impossible for realization. 
   Configurations from class $S_6$ is dismissed using similar argument. 

\begin{multicols}{2}

\begin{figure}[H]
\includegraphics[width=0.3\textwidth]{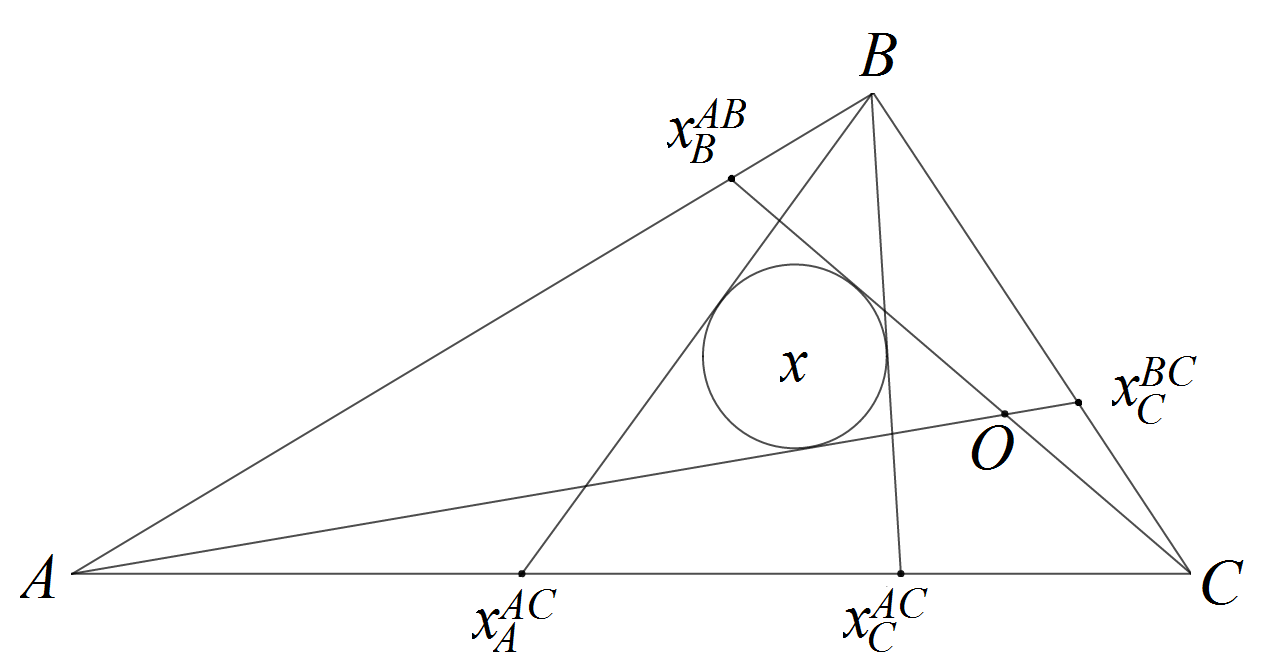}  
\caption{ }
\label{fig:Circles3}
\end{figure}

\columnbreak

\begin{figure}[H] 
\includegraphics[width=0.3\textwidth]{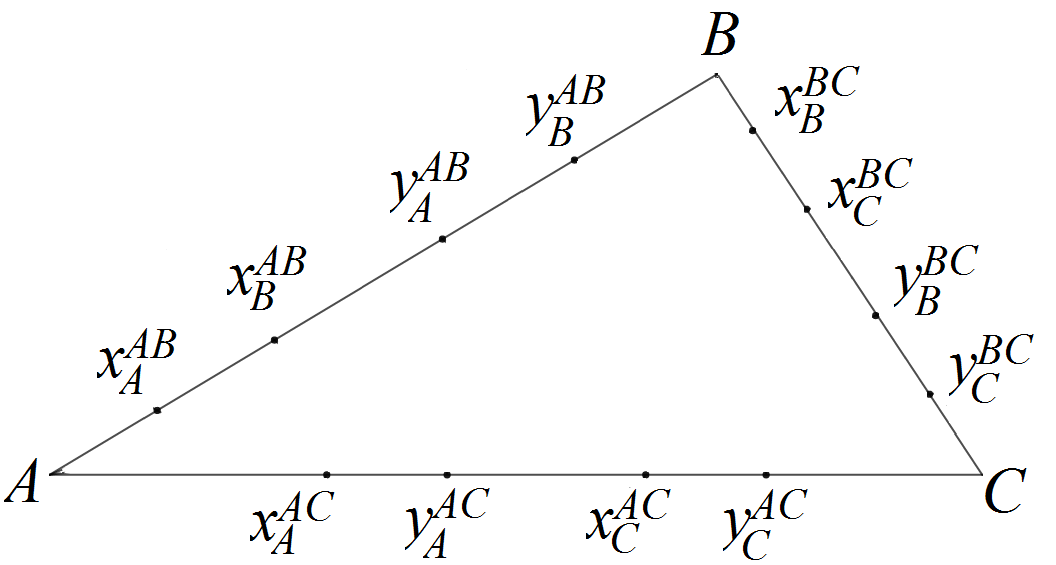}  
\caption{ }
\label{fig:DismissCase4}
\end{figure}

\end{multicols}

\end{enumerate}
 
Now consider the remaining classes of configurations. 
We illustrate realizations of non-dismissed configurations by a representative from each of classes $S_{1},S_{3},S_7,S_{11}, S_{16},S_{19},$ $S_{23},S_{27},S_{29},S_{33},S_{37}$ in Appendix \ref{App:AppendixC}. Realizations of configurations from other classes are symmetric to those illustrated. Say, configurations $C_{125}$ from $S_7$ and $C_{126}$ from $S_8$ are symmetric. So, we consider $\{S_7, S_8\}$ as a combination of classes representing symmetric realizations. 
Similarly, $\{S_3,S_4\}, \{S_{11},S_{14}\}, \{S_{16},S_{18}\}, \{S_{19},S_{20}\}$, $\{S_{23},S_{24},S_{26}\}, \{S_{29},S_{30}\}$, $\{S_{33},S_{36}\}$ represent symmetric configurations.

Thus, cases that are not dismissed are all realizable. We now want to show that  in all these realizable cases the Weak Carousel property holds.

We will need the following observation to guarantee placement of circle $y$ with respect to placement of $x$.
\begin{lemma}\label{ACN} Let $x,y$ be two circles inside $\triangle ABC$ which form configuration $C_{jkl}$, $1\leq j,k,l \leq 6$ with respect to projections to sides $AB$, $BC$ and $CA$, respectively. Consider tangent lines $(Ax_B^{BC}), (Cx_B^{AB})$ to circle $x$, and let $N= (Ax_B^{BC})\cap (Cx_B^{AB})$.
If $j \in \{1,3,5\}$ and $k\in \{2,3,6\}$, then $y$ is inside $\triangle ACN$.
\end{lemma}

\begin{proof} Note that $\triangle ACN = \triangle ACx_B^{BC} \cap \triangle ACx_B^{AB}$, so it is enough to show that, under assumption of Lemma, $y$ is inside of both $\triangle ACx_B^{BC}$ and  $\triangle ACx_B^{AB}$. Point $x_B^{AB}$ corresponds to $x_2$ in the coding of projection to $AB$, and the position of projections to $AB$ is indicated in the first index of configuration $C_{jkl}$. By assumption, $j \in \{1,3,5\}$, which implies that $y_1, y_2 \in [A,x_B^{AB}]$. This guarantees that $y$ is inside  $\triangle ACx_B^{AB}$. Point $x_B^{BC}$ corresponds to $x_1$ in the coding of projections to side $BC$, and the position of projections to $BC$ is reflected in the second index of configuration $C_{jkl}$.
By assumption, $k \in \{2,3,6\}$, which implies that $y_1,y_2 \in [x_B^{BC},C]$. This guarantees that $y$ is inside $\triangle ACx_B^{AB}$.
\end{proof}

Now pick any realized class and pick the first listed configuration there. By Lemma \ref{ACN}, we will be guaranteed that $y$ is inside $\triangle ACN$.
Check the illustration of picked configurations on Figures \ref{fig:S1} - \ref{fig:S37}.

By Corollary \ref{suff}, any of 4 sufficient conditions below would imply $y \in ch_c(x,A,C)$, which is enough  to conclude that the Weak Carousel property holds. These are 
\begin{itemize}
\item[(1)] $y_C^{AC}$ is closer to $C$ than $x_C^{AC}$;
\item[(2)] $y_A^{AC}$ is closer to $A$ than $x_A^{AC}$;
\item[(3)] $y_C^{BC}$ is closer to $C$ than $x_C^{BC}$;
\item[(4)] $y_A^{AB}$ is closer to $A$ than $x_A^{AB}$.
\end{itemize} 

We list here the conditions that hold for given configurations: $C_{121}$ from $S_1$ - (1),(3),(4); $C_{123}$ from $S_3$ - (3),(4); $C_{125}$ from $S_7$ - (1),(3); $C_{133}$ from $S_{11}$ - (4); $C_{136}$ from $S_{16}$ - (2); $C_{163}$ from $S_{19}$ - (3); $C_{165}$ from $S_{23}$ - (1),(3); $C_{335}$ from $S_{29}$ - (1); $C_{365}$ from $S_{33}$ - (1),(3); 
$C_{565} $ from $S_{37}$ - (1),(3).

It leaves only $S_{27}$, for which none of the sufficient conditions (1)-(4) holds. So we consider this case separately. Here the projections of $y$ on each side of $\triangle ABC$ are strictly inside of respective projections of $x$. Refer to Figure  \ref{fig:S27} for particular realization of this configuration. We note that $y$ can be located anywhere inside hexagon formed by 6 tangent lines to circle $x$ from $A,B,C$. 
If $y$ has no intersection with area $w_N$ defined in Lemma \ref{wN}, then $y \in ch_c(x,A,C)$.

If $y$ is located so that it has non-empty intersection with area $w_N$, then by Corollary \ref{corollary4}, $y$ would have no intersection with area $w_{AC}$, which is disjoint from $w_N$ outside $x$ and inside $\triangle ACN$. In particular, no intersection with area $w_M\subseteq w_{AC}$ in the corner of hexagon next to point $M= (Bx_C^{AC})\cap (Ax_C^{BC})$, outside circle $x$ and inside $\triangle ABM$. Thus, $y$ is inside $\triangle ABM$ and avoids area $w_M$, therefore, $y \in ch_c(x,A,B)$. We conclude that, in any case, $y$ is in the closure of $x$ and two vertices of $\triangle ABC$, which is needed.

Now, we consider configurations when circles $x$ and $y$ may have coincident projection points. We prove the property when $x$ and $y$ have at least one common tangent line through one of the vertices, say, projection points $x_C^{BC}$ and $y_C^{BC}$ coincide . 

\begin{itemize}
\item[Case 1.] $x$ and $y$ are on one side of a semi-plane made by the common tangent line. Say, $x$ and $y$ are in $\triangle ABx_C^{BC}$. 
Wlog consider configurations with $|Ax_C^{AC}|\leq |Ay_C^{AC}|$. Let $N$ be a point of intersection of tangent lines from $A$ and $B$ to $y$ as shown on Figure \ref{fig:coinc1ab}(a). Then $x$ is in $\triangle{ABN}$. 

Let $w_1$ and $w_2$ be two disjoint areas in $\triangle{ABN}$. By Corollary \ref{corollary4}, $x$ may have intersection with only one of areas $w_1,w_2$. If $x$ is in $\triangle{ABN} \backslash w_2$ as in Figure \ref{fig:coinc2}, then $x\in ch_{c}(\{y, A, B \})$. If $x$ is in $\triangle{ABN} \backslash w_1$ as in Figure \ref{fig:coinc3}, then $x\in ch_{c}(\{y, B, C \})$. Indeed, $  w_2 \subseteq \CHull(\tilde{y}\cup\{C\})$ by Lemma \ref{wN}, see Figure \ref{fig:coinc1ab}(b).

\begin{figure}[H]
\begin{center}
\subfigure[ ]{
\includegraphics[width=0.3\textwidth]{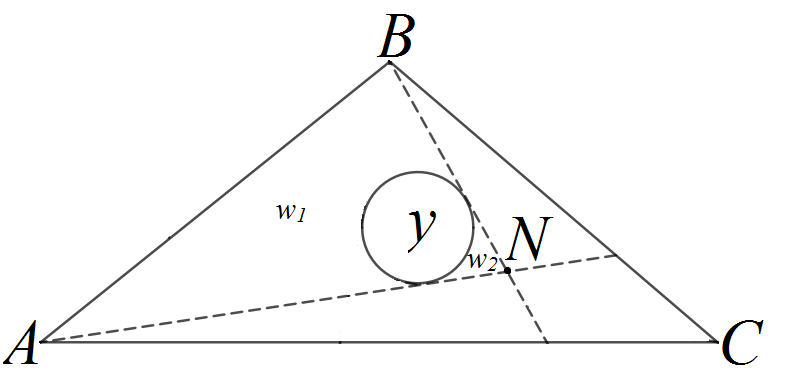}%
}%
\hspace{3ex}
\subfigure[ ]{
\includegraphics[width=0.3\textwidth]{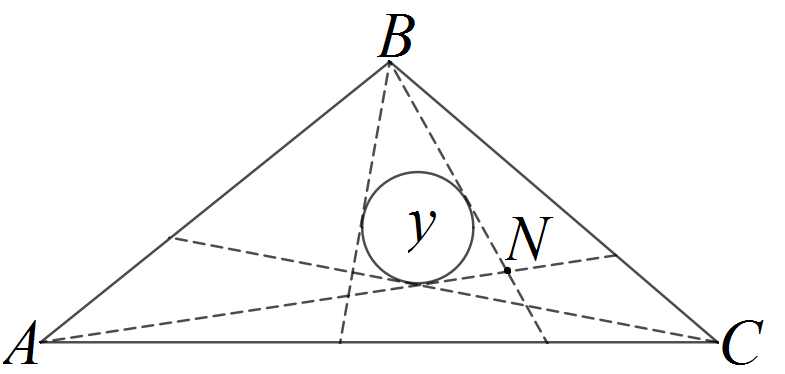}%
} 
\caption{ }
\label{fig:coinc1ab}
\end{center}
\end{figure}

\begin{multicols}{2}
\begin{figure}[H]
\includegraphics[width=0.3\textwidth]{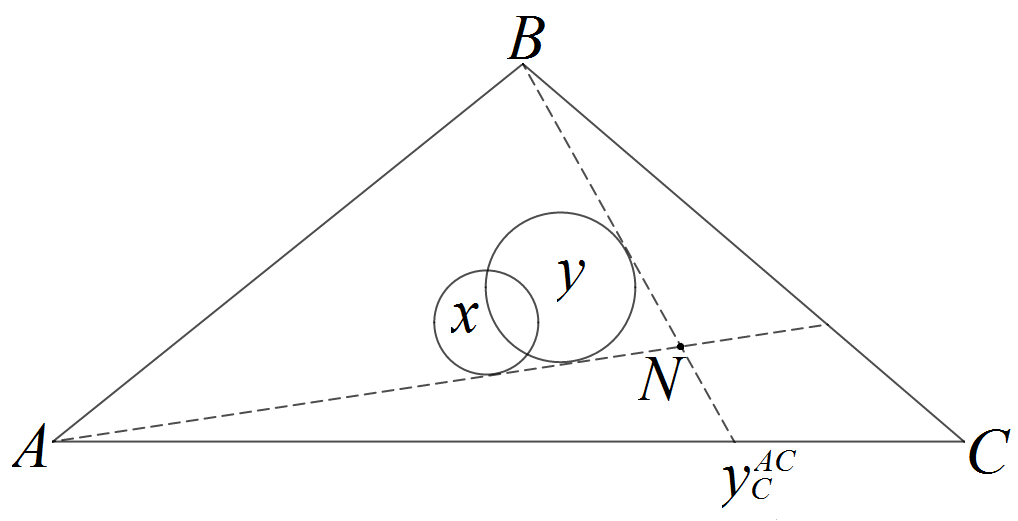}  
  \caption{}
  \label{fig:coinc2}
\end{figure}
\columnbreak
\begin{figure}[H]
\includegraphics[width=0.3\textwidth]{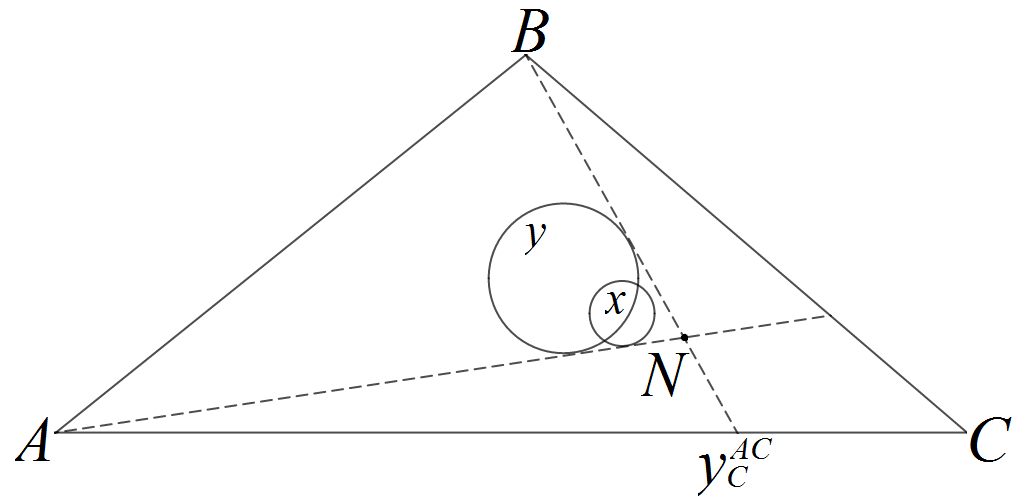}  
  \caption{}
  \label{fig:coinc3}
\end{figure}
\end{multicols}

\item[Case 2.] $x$ and $y$ are on opposite sides of a semi-plane made by the common tangent line. Say, $x$ is in $\triangle Ax_C^{BC}B$ and $y$ is in 
$\triangle Ax_C^{BC}C$.

\begin{multicols}{2}
\begin{figure}[H]
\includegraphics[width=0.3\textwidth]{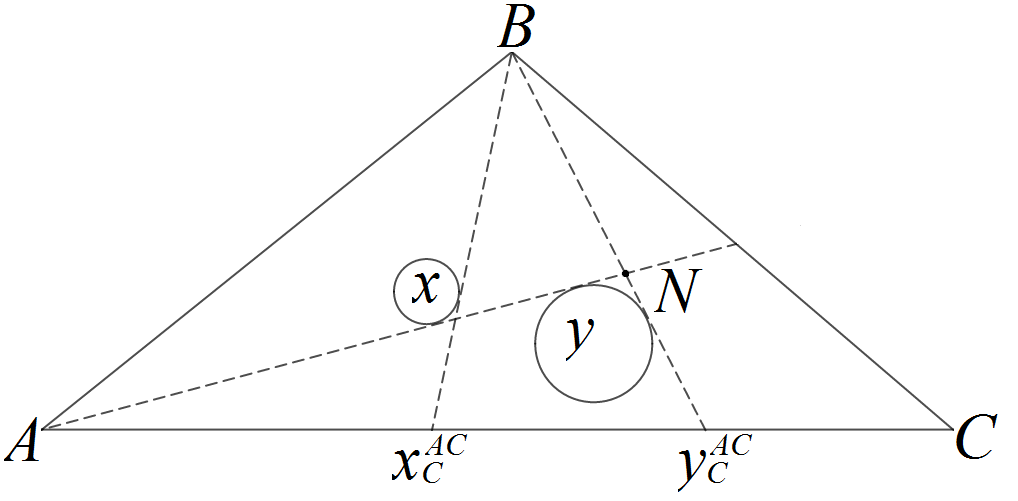}  
  \caption{}
  \label{fig:coincid1}
\end{figure}
\begin{figure}[H]
\includegraphics[width=0.3\textwidth]{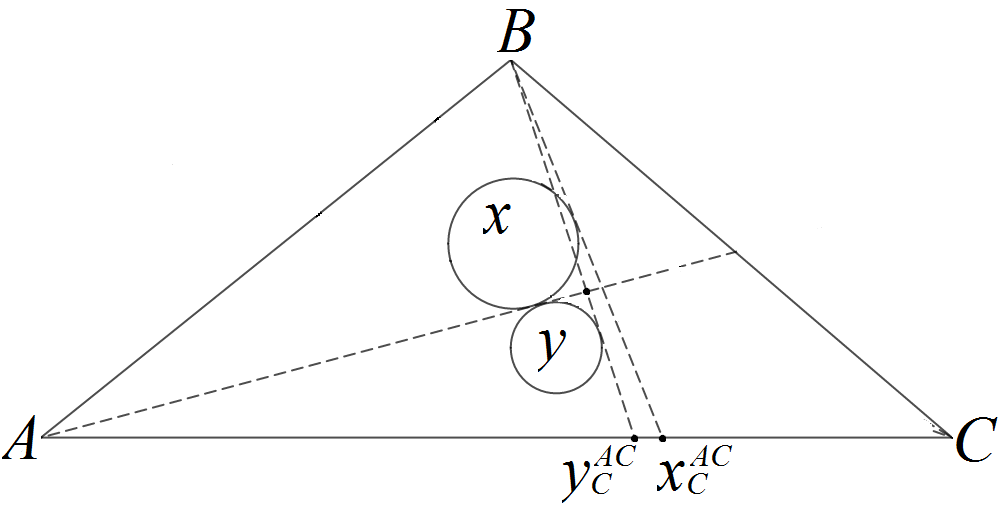}  
\caption{}
  \label{fig:coincid2}
\end{figure}
\end{multicols}

If $|Ax_C^{AC}|\leq |Ay_C^{AC}|$ as in Figure \ref{fig:coincid1}, then $x\in ch_{c}(\{A, B, y\})$. 
If $|Ax_C^{AC}|>|Ay_C^{AC}|$ as in Figure \ref{fig:coincid2}, then $y \in  ch_{c}(\{A, x_C^{AC}, x\})\subseteq ch_{c}(\{A,C,x\})$, which is needed.

\end{itemize} 
\end{proof}

\section{Lemmas} \label{Lemmas}

\begin{lemma}\label{lemma4.2}
Let $p$ be a circle inscribed in $\angle{BAC}$ and 
let $w_1$ denote the area of $\angle{BAC}$ outside circle $p$ and adjacent to vertex $A$, see Fiugure \ref{fig:Lemma5}. If circle $y$ is inside $\angle{BAC}$ and has a point inside which also belongs to $w_1$, then $y \subseteq \CHull (\tilde{p} \cup \{A\})$.
\end{lemma}
\begin{proof}

Let $F$ be the center of circle $y$. Suppose the distance from $F$ to line $(AB)$ is less than or equal to the distance from $F$ to line $(AC)$. Then $y$ is inside another circle $g$, with center $F$ and tangent line $(AB)$; moreover, $g \subseteq \angle{BAC}$, see Figure \ref{fig:Lemma42}. Let circle $s$ share with $g$ its touching point $Q \in (AB)$, and let it have its center $O_s$ on ray $[A,O_p)$,  where $O_p$ is the center of circle $p$. Then $g\subseteq s$; moreover, circle $s$ is inscribed into $\angle BAC$. 

 \begin{figure}[H]
\includegraphics[width=0.3\textwidth]{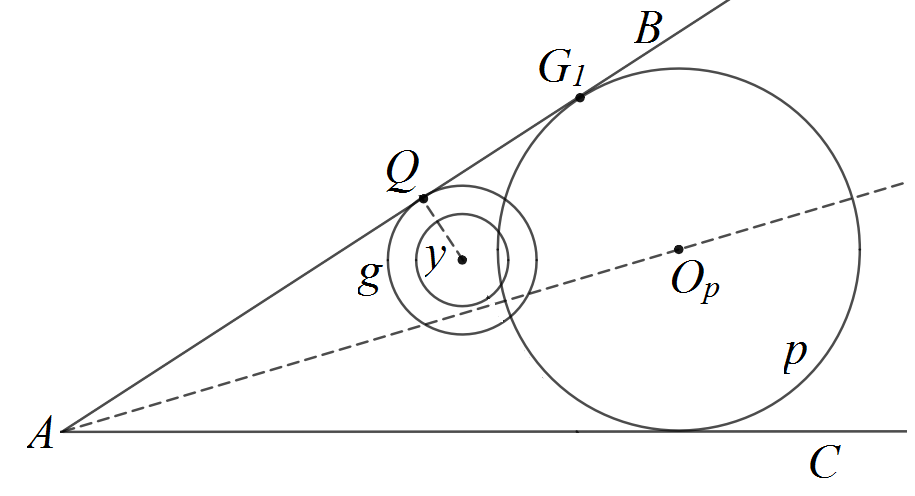}  
\caption{ }
  \label{fig:Lemma42}
\end{figure}
Note that circle $s$ is always homothetic to circle $p$, with the center of homothety $A$ and the coefficient of homothety $k=\frac{|A O_s|}{|AO_p|}$.

There are two possibilities: 

(1) $k >1$, in which case $s$ does not have any point of intersection with $w_1$, see Figure \ref{fig:Lemma422}. Indeed, if point $E$ is inside $s$ and belongs to $w_1$, then there should be point $E'$ inside $p$ and homothetic to $E$. Then $E' \in [A,E]$, hence, $E'$ is also in $w_1$, which contradicts the assumption that $w_1$ is outside $p$. 

(2) $k\leq 1$, which means that $O_s$ belongs to segment $[A,O_p]$, see Figure \ref{fig:Lemma421}. Then every point $M_s$ of circle $s$ is on the segment $[A,M_p]$, where point $M_p$ is in $p$ and homothetic to $M_s$. It follows that circle $s$ is in the convex closure of circle $p$ and point $A$. This implies  $y \subseteq g\subseteq s\subseteq \CHull (\tilde{p} \cup \{A\})$, which is needed.

\begin{multicols}{2}

 \begin{figure}[H]
\includegraphics[width=0.3\textwidth]{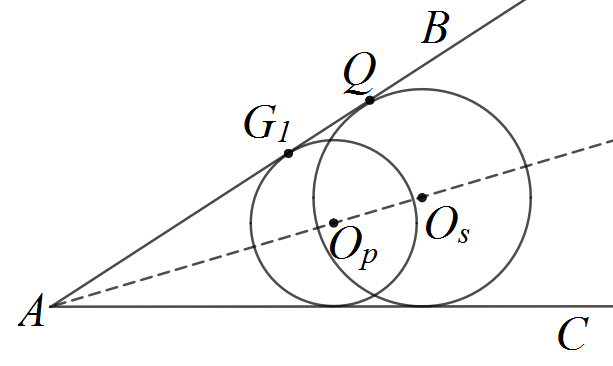}  
\caption{ }
  \label{fig:Lemma422}
\end{figure} 

\columnbreak

\begin{figure}[H]
\includegraphics[width=0.3\textwidth]{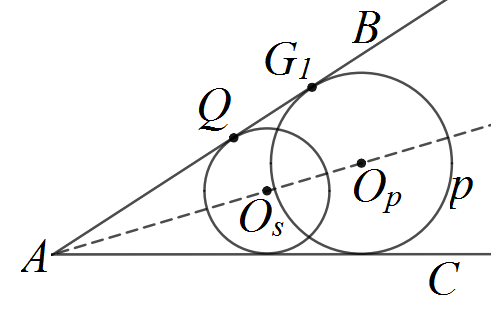}  
\caption{ }
  \label{fig:Lemma421}
\end{figure}
\end{multicols}
\end{proof}

\begin{corollary}\label{corollary4}
Suppose $p$ is a circle inscribed in $\angle{A}$ of triangle $\triangle{ABC}$. Let $w_1$ be area of $\triangle{ABC}$ outside circle $p$ and adjacent to vertex $A$, and $w_2= \triangle{ABC}\setminus (p \cup w_1)$ as in Figure \ref{fig:Lemma5}. Then any circle $y\subseteq \triangle{ABC}$ cannot have non-empty intersections with both areas $w_1,w_2$.
 \end{corollary}

\begin{proof} Since $y$ is inside $\angle BAC$, one can apply Lemma \ref{lemma4.2} and claim, that if $y$ has a point in area $w_1$, then $y \subseteq 
\CHull(\tilde{p}\cup \{A\})$, therefore, $y$ does not have common points with $w_2$.
\end{proof}

 \begin{figure}[H]
\includegraphics[width=0.3\textwidth]{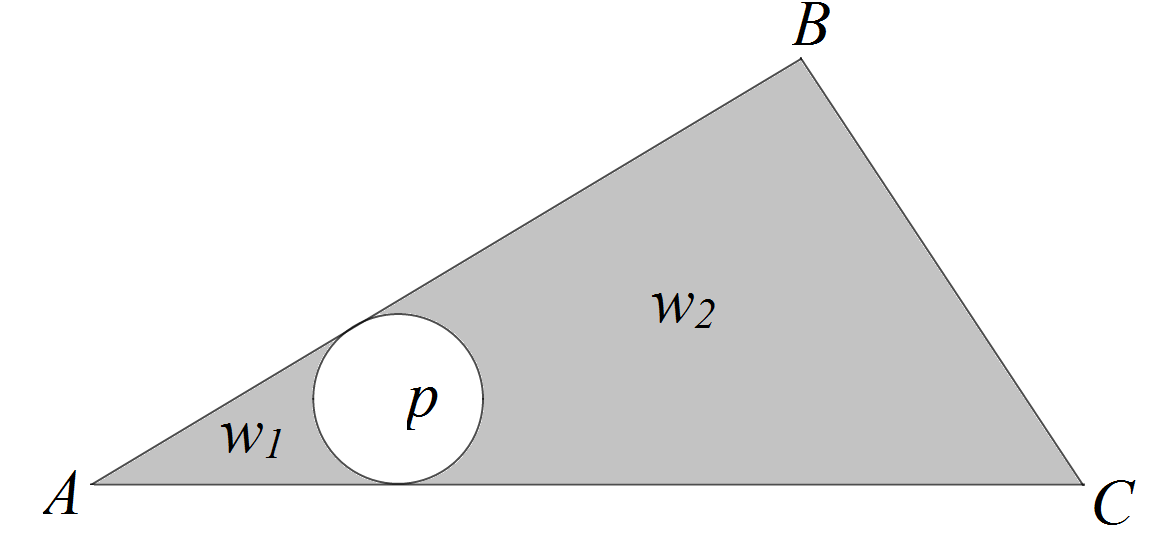}  
\caption{ }
  \label{fig:Lemma5}
\end{figure}

The following couple lemmas are elementary statements from plane geometry. We provide the proofs for completeness. 

     \begin{lemma}\label{lemma6}
Suppose $s$ is a circle inscribed in $\angle{A}<\pi$. Circle $s$ divides $\angle{A}$ in two disjoint areas $w_1$ and $w_2$. Then, any non-parallel tangent lines: to any point of $s$ on the border with $w_1$ and to any point of $s$ on the border with $w_2$ - intersect in a point outside the interior area of $\angle{A}$.
   \end{lemma}

   \begin{proof}[Proof]
   
\begin{multicols}{2}
 \begin{figure}[H]
 \includegraphics[width=0.2\textwidth]{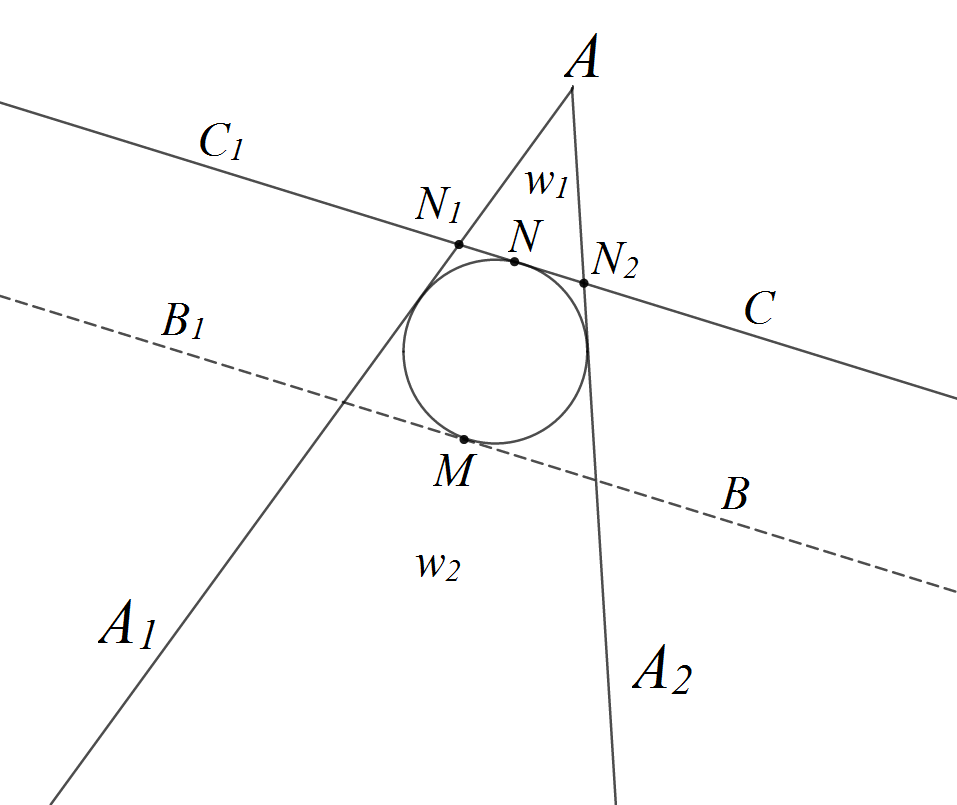}
 \caption{ }
 \label{fig:Lemma6}
 \end{figure}
\columnbreak 
 \begin{figure}[H]
 \includegraphics[width=0.2\textwidth]{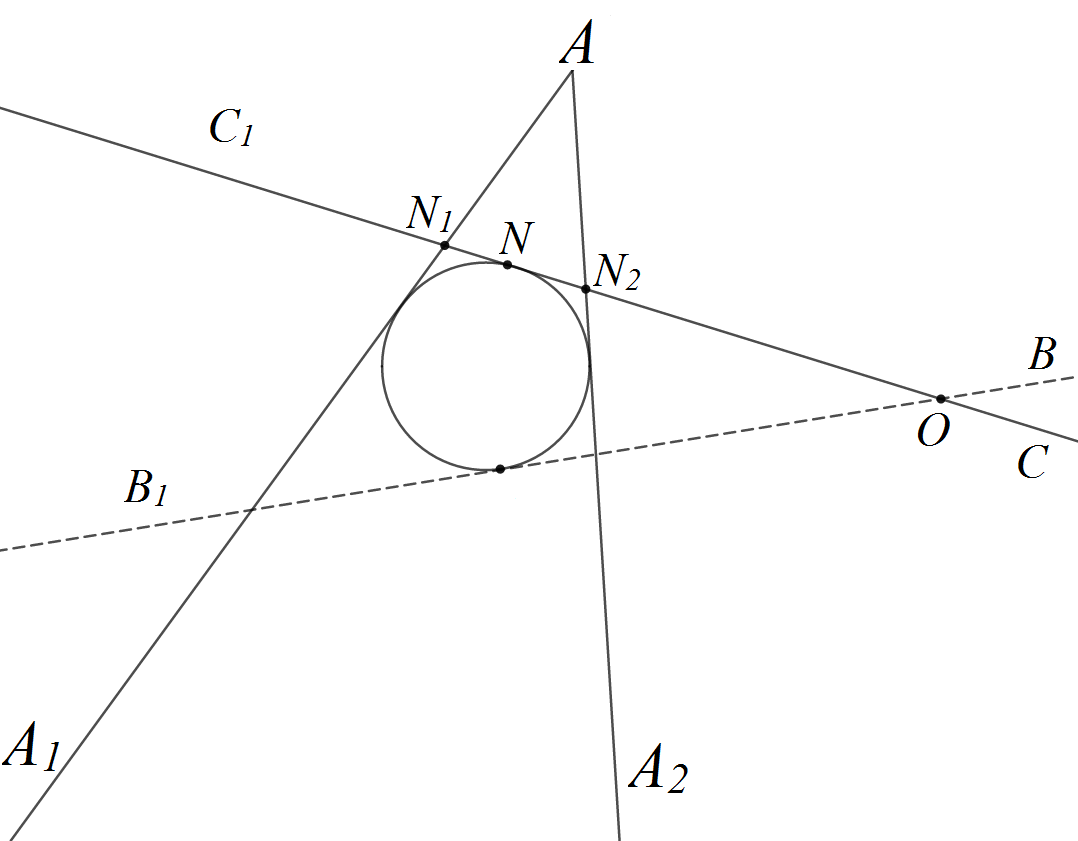}
 \caption{ }
 \label{fig:Lemma6intersect}
 \end{figure}
 \end{multicols}

Call the lines forming angle $\angle{A}$ as $(AA_1)$ and $(AA_2)$. Let $(CC_1)$ be a tangent line to $s$ at a point $N$ which is located on the border of $s$ with $w_1$. Let points of intersection of $(CC_1)$ with $(AA_1)$ and $(AA_2)$ be $N_1$ and $N_2$ respectively. 
Denote by $M$ the other end point of a diameter starting at $N$, see Figure \ref{fig:Lemma6}. 
Since the border of $s$ with $w_1$ is a strictly smaller arc of $s$ than the arc of border with $w_2$, point $M$ belongs the latter arc.
 
Let $(BB_1)$ be a tangent line to $s$ which touches $s$ at some point on the border with $w_2$. We will fix $(CC_1)$, and will treat line $(BB_1)$ as rotating from its initial position as line $(AA_1)$ and toward position when it is parallel to line $(CC_1)$, so that the touching point with $s$ moves from $N_1$ to $M$. If $O$ denotes the point of intersection of $(BB_1)$ and $(CC_1)$ as on Figure \ref{fig:Lemma6intersect} , then $O$ moves from initial location at $N_1$ along the ray $[N_1)$ of $(CC_1)$, which does not contain point $N$. 

Similarly, if $(BB_1)$ starts as $(AA_2)$ and rotates as the touching point with $s$ moves from $N_2$ to $M$, point $O$ moves from $N_2$ along the ray $[N_2)$ of line $(CC_1)$, which does not contain point $N$.

Thus, the point of intersection of tangent lines $(BB_1)$ and $(CC_1)$ is always outside the interior area of $\angle{A}$.

\end{proof}

\begin{lemma}\label{wN} Suppose circle $x$ is inside $\triangle ABC$, and its tangent lines $(Ax_B^{BC})$ and $(Cx_B^{BA})$ intersect at point $N$. Call area that complements disc defined by $x$ in $\CHull (\tilde{x} \cup\{N\})$ as $w_N$. Then $w_N \subseteq \CHull (\tilde{x} \cup\{B\})$.
\end{lemma}

\begin{proof} It is easy to observe that $B$ is located in smaller of two angles formed by rays $[N,x_B^{BA})$ and $[N,x_B^{BC})$, see Figure \ref{fig:Proof4.5}. For each point located in that angle, the tangent line through that point to the circle does not have points in the border area with $w_N$. Indeed, one can easily check that any tangent line to $x$ at any point belonging to border with $w_N$ is completely inside the larger angle formed by $[N,x_B^{BA})$ and $[N,x_B^{BC})$.

Thus, we have touching points $B_1,B_2$, for the tangent lines from $B$ to circle $x$, on the arc not bordering $w_N$. As a result, arc bordering $w_N$ is a subarc of a smaller of two arcs formed by $B_1$ and $B_2$.

 \begin{figure}[H]
\includegraphics[width=0.2\textwidth]{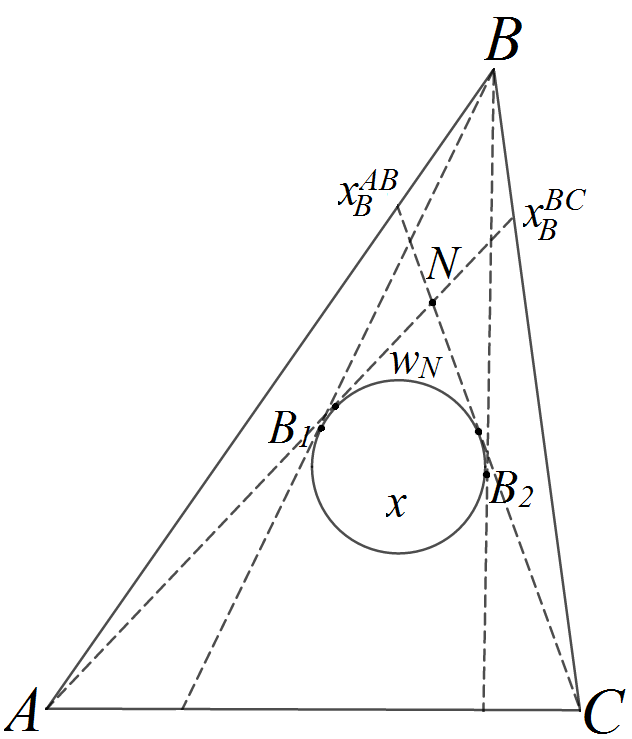}  
\caption{ }
  \label{fig:Proof4.5}
\end{figure}

We claim that any two tangent lines touching circle at two points on the smaller arc $(B_1B_2)$ intersect inside $\triangle B_1BB_2$. In particular, we get $N$ inside this triangle. In order to prove the claim, one could fix one of the tangent lines, say at point $U$ on the smaller arc $(B_1B_2)$. Touching point of line $(Ax_B^{BC})$ can play a role of $U$ on Figure \ref{fig:Proof4.5}. Choose the second line at location of line $(BB_1)$, or line $(BB_2)$, then move this line so that its touching point travels along the smaller arc $(B_1B_2)$, either from $B_1$ to $U$, or from $B_2$ to $U$. Then the point of intersection travels along the segment of the first tangent line inside $\triangle B_1BB_2$.

It follows that the arc bordering $w_N$ as well as point $N$ are inside $\triangle B_1BB_2$, from where $w_N \subseteq \CHull (\tilde{x} \cup\{B\})$ follows. 
\end{proof}

\begin{corollary}\label{suff} Suppose circles $x,y$ are inside $\triangle ABC$, two touching lines to $x$, $(Ax_B^{BC})$ and $(Cx_B^{BA})$, intersect at point $N$, and $y$ is inside $\triangle ACN$. Then any of the following conditions is sufficient for $y \in ch_c (x,A,C)$:
\begin{itemize}
\item[(1)] $y_C^{AC}$ is closer to $C$ than $x_C^{AC}$;
\item[(2)] $y_A^{AC}$ is closer to $A$ than $x_A^{AC}$;
\item[(3)] $y_C^{BC}$ is closer to $C$ than $x_C^{BC}$;
\item[(4)] $y_A^{AB}$ is closer to $A$ than $x_A^{AB}$.
\end{itemize} 
\end{corollary}
\begin{proof} Let $w_N$ be an area inside $\triangle ACN$ defined in Lemma \ref{wN}, and $w_{AC}$ is the disjoint from $w_N$ area outside $x$ and inside $\triangle ACN$. According to Corollary \ref{corollary4} circle $y$ may have non-empty intersection only with one of areas $w_N$, $w_{AC}$. It follows from Lemma \ref{wN} that in case (1) or (2), $y$ cannot avoid intersecting with $w_{AC}$. Indeed, otherwise, $y \subseteq  \CHull (\tilde{x} \cup\{B\})$, and projection of $y$ onto $AC$ would be inside projection of $x$ onto $AC$, which contradicts the assumption in cases (1) and (2). Therefore, $y\cap w_N = \emptyset$ and  $y \in ch_c (x,A,C)$.

Evidently, $w_N \subseteq \triangle Ax_C^{BC} B$, therefore, in case (3) $y$ should have intersection with $w_{AC}$. Similarly, $w_N \subseteq \triangle Cx_A^{AB}B$, therefore, in case (4) $y$ should also have intersection with $w_{AC}$. Hence, in any of these cases we get the same conclusion:  $y \in ch_c (x,A,C)$.  
\end{proof}

\section{Weak $2\times 3$-Carousel rule for a geometry of circles on a plane} \label{sectionWCP}

   \begin{theorem}\label{theorem2}
   Every convex geometry of circles $(F, ch_c)$ in $\mathbb{R}^2$ satisfies the Weak $2\times 3$-Carousel rule. 
   \end{theorem}
   
  \begin{proof}[Proof]
   \begin{figure}[H]
\includegraphics[width=0.3\textwidth]{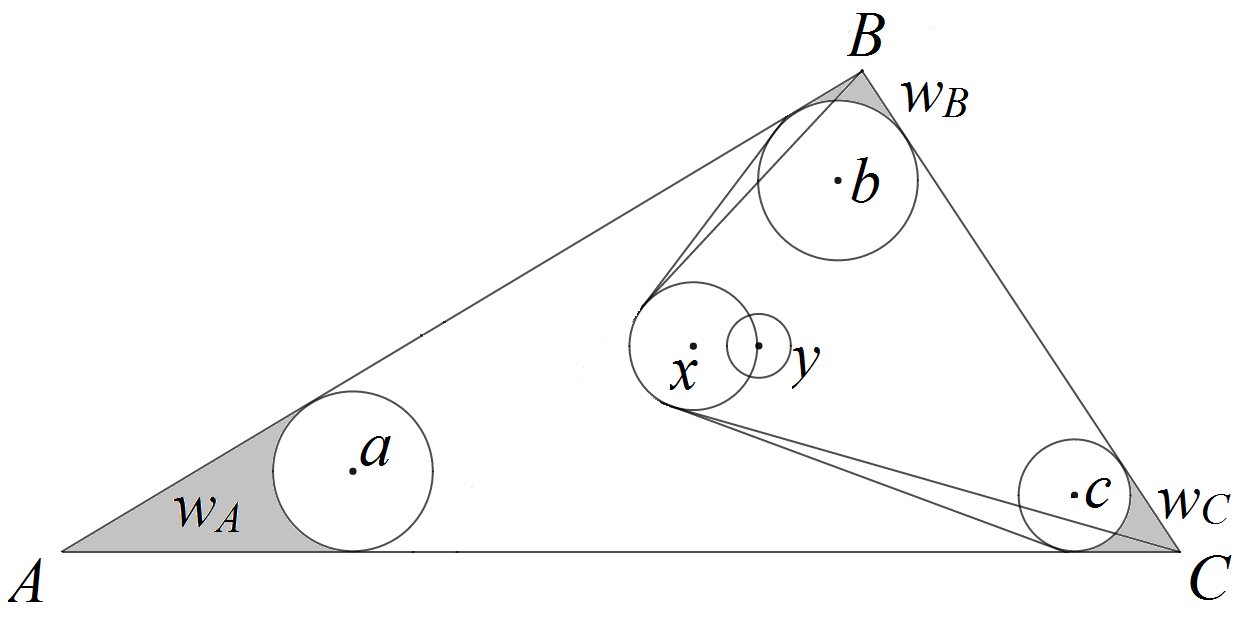}  
\caption{ }
  \label{fig:WeakCarousel}
\end{figure}

Consider circles $x,y, a,b,c \in F$ such that $x, y \in ch_c(\{a,b,c\})$. Then $\tilde{x}, \tilde{y} \subseteq \CHull (\tilde{a}\cup \tilde{b}\cup\tilde{c})$,  see Figure \ref{fig:WeakCarousel}.

If $\CHull (\tilde{a}\cup \tilde{b}\cup\tilde{c}) = \CHull (\tilde{a}\cup \tilde{b})$, then conclusion of the statement obviously holds. Thus, we may assume that none of the circles is in the convex hull of two others.

Draw common tangent lines to pairs of circles $a, b, c$ to obtain a triangle where circles are inscribed into three angles, let the vertices of triangle be $A, B$ and $C$. Denote $w_A, w_B, w_C$ disjoint areas inside $\Delta ABC$ and outside  $\CHull (\tilde{a}\cup \tilde{b}\cup\tilde{c})$, as shown on the picture. 
Then $ch_c(\{a,b,c\})=\CHull \{A,B,C\}\backslash (w_A\cup w_B\cup w_C)$, so $\tilde{x}, \tilde{y} \subseteq \CHull\{A,B,C\}$.
 
By Theorem \ref{theorem1}, either $x$ is in a convex hull of two points from $S$ and $y$, or $y$ is in a convex hull of two points from $S$ and $x$. 

Say, $\tilde{y}\subseteq  \CHull (\tilde{x} \cup\{B,C\})$. Apparently, $y \cap w_B = \emptyset = y\cap w_C$, and $\CHull (\tilde{x} \cup\{B,C\})\setminus (w_B \cup w_C) \subseteq  \CHull (\tilde{x}\cup \tilde{b}\cup \tilde{c})$. Therefore, $\tilde{y}\subseteq  \CHull (\tilde{x} \cup\tilde{b}\cup \tilde{c})$, or $y \in ch_c( \{x,b,c\})$, which is needed. 

It follows that the Weak $2\times3$-Carousel rule holds for convex geometry of circles.  

\end{proof}

\section {Example of convex geometry that is not representable by circles}\label{Example}

We borrow examples of convex geometries in this section from \cite{Adaricheva}. 
To simplify notation, sets of points $\{x,y,z\}$ will be denoted simply $xyz$ etc.

Consider an example of affine convex geometry $G'=(X, \mathcal{F'})$, where \\$X=\{a_0,a_1,a_2, x, y\}$ is a set of points on a plane as shown in Figure \ref{fig:ExampleAffine}. Then, $\mathcal{F'}=\mathcal{P}(X)\backslash \{a_0a_1a_2, a_0a_2x, a_0a_1y, a_0a_1a_2x, a_0a_1a_2y \}$. 

\begin{figure}[h!]
\includegraphics[width=0.3\textwidth]{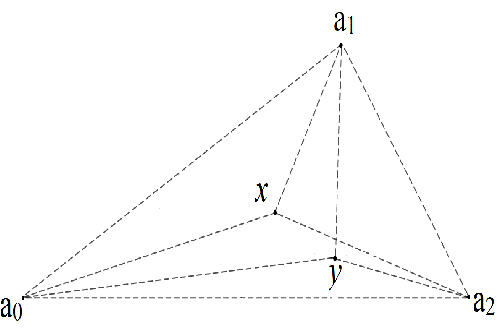}  
\caption{ }
  \label{fig:ExampleAffine}
\end{figure}

Consider $G=(X, \mathcal{F})$, where $\mathcal{F}=\mathcal{F}' \cup \{a_0a_2x, a_0a_1y\}= \mathcal{P}(X)\backslash \{ a_0a_1a_2, a_0a_1a_2x, a_0a_1a_2y\}$.

It could be directly checked  from Definition \ref{Alignment} that $\mathcal{F}$  satisfies the properties of alignments on ground set $X$. Initial alignment $\mathcal{F'}$ satisfies the required properties since $G'$ is affine convex geometry. 
All subsets of added sets are already in $\mathcal{F}'$, in particular their intersections with other sets from $\mathcal{F}$ are also in $\mathcal{F}$. Moreover, one-point extensions for both new sets exist in $\mathcal{F}'$. So, $G$ is, indeed, a convex geometry. 

From \cite{Adaricheva}, $G=(X, \mathcal{F})$ does not satisfy the 2-Carousel rule. It was shown in \cite{Adaricheva} that convex geometries failing $n$-Carousel rule could not be weakly represented by affine convex geometries in $\mathbb{R}^n$. So, $G$ could not be weakly represented by affine convex geometries in $\mathbb{R}^2$. Now, we show that $G$ does not also satisfy the Weak 2-Carousel rule. 

Let $\varphi : 2^X \rightarrow 2^X$ be a corresponding closure operator for alignment $F$ on $X$.

Then, $x, y\in \varphi (a_0a_1a_2)$ since $\varphi (a_0a_1a_2)=X$. Moreover, 
$\varphi (a_ia_jx)=a_ia_jx$ and $\varphi (a_ia_jy)=a_ia_jy$, for $i,j=0,1,2$.

It follows that neither $x$ nor $y$ is in the closure of the other together with any two members of $\{a_0,a_1,a_2\}$. Hence, $G$ does not satisfy the Weak $2\times 3$-Carousel rule, and thus, it also fails the Weak $2$-Carousel rule.  

We prove in Section \ref{sectionWCP} that any geometry of circles satisfies the Weak $2\times 3$-Carousel rule (Theorem \ref{theorem2}). Therefore, $G$ cannot be represented by circles on a plane.    

\begin{remark}
The given example is minimal in the cardinality of the base set as well as the cardinality of its alignment, for which the Weak $2\times 3$-Carousel rule fails. 
\end{remark}
Indeed, the failure of the rule assumes that we have at least 5 elements in the base set with  $x, y\in \varphi (a_0a_1a_2)$, while a closure of any of $x,y$ with two points among $a_0,a_1,a_2$ will not contain second point. Given example is a closure system defined by just one \emph{implication} $a_0a_1a_2\rightarrow xy$, in other words, there are only three subsets of the base set: $\{a_0,a_1,a_2\}$, $\{a_0,a_1,a_2, x\}$ and $\{a_0,a_1,a_2, y\}$ - for which the closure is strictly larger than itself. The alignment of any other closure system with five elements and implication $a_0a_1a_2\rightarrow xy$ will be a subset of alignment of above example, and it should have more implications in its basis. It is easy to verify that, with any other implication in the basis, one of implications $xa_ia_j \rightarrow y$ or $ya_ia_j \rightarrow x$ will also hold in the closure system.

\section{Concluding Remarks}\label{CR}

We demonstrated an example of convex geometry of $\cdim=6$ that fails the Weak $2\times 3$-Carousel rule and therefore could not be represented by circles on a plane. 
This convex geometry fails the 2-Carousel property, so it is also not weakly represented by affine convex geometries on a plane (\cite{Adaricheva}). 
Hence, we ask the following problem:

\begin{problem}\label{1st}
Is every convex geometry of $\cdim=3,4$ or $5$ strongly represented by a geometry of circles on a plane? 
\end{problem}

Similarly, we do not know whether such geometries have a weak representation on a plane.\\

Since we showed the existence of geometries that are not strongly representable by circles in $\mathbb{R}^2$, we would like to consider higher dimensions of space. Therefore, we ask the following:

\begin{problem}\label{2d}
Can every finite convex geometry be strongly represented by balls in $\mathbb{R}^n$?
\end{problem}
 
In \cite{RichRog}, parameter $\dim (G)$ for convex geometry $G$ indicates the smallest dimension $n$ of space $\mathbb{R}^n$, for which a weak representation of $G$ exists. Similarly, we can define $\dim_c (G)$, a smallest dimension of space for which the strong representation by circles exists.

\begin{problem}\label{3d}
For any geometry $G$ that has strong representation by the balls,  
is it true that
$\dim_c(G) \leq \dim (G)$?
\end{problem}

Consider an example of a convex geometry $G$ for which $\dim_c(G) < \dim (G)$. Take a convex geometry of circles with elements of the base set $x,y, A,B,C$ as in Figure \ref{fig:S3}. Then, of course, the Weak $2\times 3$-Carousel rule holds, but the $2$-Carousel property fails, because $x \not \in ch_c (y,A,B) \cup ch_c (y,B,C) \cup ch_c (y,A,C)$. Then, according to the result of \cite{Adaricheva}, this convex geometry does not have a weak representation by points on a plane. Therefore, $\dim_c(G) = 2 < \dim (G)$. 

Even in cases when $\dim_c(G) = \dim (G)$, the weak representation may imply the need for a larger base of geometry $G'$ for the weak representation. For example, consider  \emph{atomistic} convex geometry on the base set $X=\{a,b,x,y\}$ defined by a single implication $ab\rightarrow xy$. Then it has a weak representation by points and strong representation by circles, both on a plane. But a geometry of points on a plane $G'$ for a weak representation will require at least 5 points in the base set.

Figure \ref{fig:CirclesExStrong} shows its representation by circles on a plane. It is easy to check that this geometry cannot be strongly represented by affine convex geometry on a plane (Points representing $x$ and $y$ should be on one line between points representing $a$ and $b$, but in this case $x$ is in a closure of, say, $\{a,y\}$ and $y$ is in a closure of $\{b,x\}$ (or, similarly, with $a$ and $b$ switched), which is not true for the geometry defined by $ab\rightarrow xy$). It can be verified that this geometry may still be weakly presented by a 5-point configuration on the plane. 

\begin{figure}[H]
\includegraphics[width=0.3\textwidth]{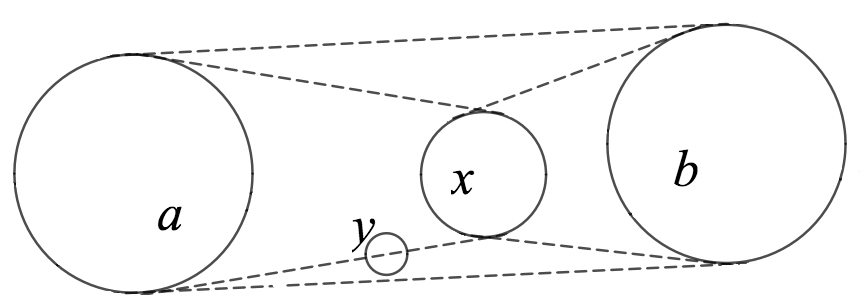}  
\caption{ }
  \label{fig:CirclesExStrong}
\end{figure}

Finally, we note that the geometry of circles introduced in \cite{Czedli} can be generalized to geometry based on other convex 2-dimensional shapes. Namely, the following statement follows the same proof as \cite[Proposition 2.1]{Czedli}.

\begin{proposition} Consider bounded convex shapes on a plane whose borders are defined by family of equations $f(u,v)=0$, with finitely many parameters and smooth function $f$, such that finitely many points on the curve may recover the parameters fully. Take finitely many of such shapes defined by equations from the same family, and define the closure operator on this base set similar to $\varphi_c$ for circles. Then the closure system will be a convex geometry. 
\end{proposition}  

For circles, $f(u,v) = u^2+v^2 - r^2$ with parameter $r$, and three points on the curve recover parameter $r$. For ellipse, $f(u,v)=\frac{u^2}{a^2} + \frac{v^2}{b^2} - 1$ with parameters $a,b$, and any 5 points on the curve can recover the parameters, thus, the curve itself.

The result is not true in general, though, for the shapes whose border is not smooth: for example, it will fail for polygons. Related results in case of segments were established in \cite{Adar2004}.

Further generalizations for representation of convex geometries by (convex) shapes can be foreseen in the future, and we may consider them in our follow-up paper.

\section{Acknowledgements}
We are grateful for the financial support of Nazarbayev University to travel to SIAM Discrete Mathematics conference in Atlanta, US, June 6-10, 2016, where the results of this paper were presented. We were helped by M. K. Adarichev in editing the paper.

\newpage
\section{Appendices}
\appendix

\section{}\label{App:AppendixA}

 \begin{figure}[H]
\includegraphics[scale=0.3]{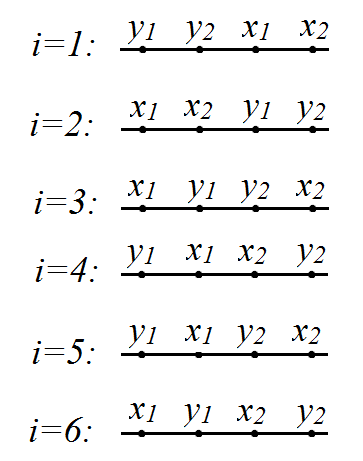}  
\caption{ }
  \label{fig:Segments}
\end{figure}

\section{}\label{App:AppendixB}

\begin{multicols}{2}

$S_1=\{C_{121},C_{122},C_{112},C_{221},C_{211},C_{212}\}$
\\$S_2=\{C_{111},C_{222}\}$

$S_3=\{C_{123},C_{142},C_{214},C_{231}, C_{312},C_{421}\}$
\\$S_4=\{C_{124},C_{132},C_{213},C_{241}, C_{321},C_{412}\}$

$S_5=\{C_{113},C_{131},C_{224},C_{242}, C_{311},C_{422}\}$
\\$S_6=\{C_{114},C_{141},C_{223},C_{232}, C_{322},C_{411}\}$

$S_7=\{C_{125},C_{162},C_{216},C_{251},C_{512},C_{621}\}$
\\$S_8=\{C_{126},C_{152},C_{215},C_{261},C_{521},C_{612}\}$

$S_9=\{C_{115},C_{151},C_{226},C_{262},C_{511},C_{622}\}$
\\$S_{10}=\{C_{116},C_{161},C_{225},C_{252},C_{522},C_{611}\}$

$S_{11}=\{C_{133},C_{244},C_{313},C_{331},C_{424},C_{442}\}$
\\$S_{12}=\{C_{134},C_{243},C_{324},C_{341},C_{413},C_{432}\}$

$S_{13}=\{C_{143},C_{234},C_{314},C_{342},C_{423},C_{431}\}$
\\$S_{14}=\{C_{144},C_{233},C_{323},C_{332},C_{414},C_{441}\}$

$S_{15}=\{C_{135},C_{246},C_{351},C_{462},C_{513},C_{624}\}$
\\$S_{16}=\{C_{136},C_{245},C_{361},C_{452},C_{524},C_{613}\}$

$S_{17}=\{C_{145},C_{236},C_{362},C_{451},C_{514},C_{623}\}$
\\$S_{18}=\{C_{146},C_{235},C_{352},C_{461},C_{523},C_{614}\}$

$S_{19}=\{C_{163},C_{254},C_{316},C_{425},C_{542},C_{631}\}$
\columnbreak

$S_{20}=\{C_{164},C_{253},C_{325},C_{416},C_{532},C_{641}\}$
\\$S_{21}=\{C_{153},C_{264},C_{315},C_{426},C_{531},C_{642}\}$

$S_{22}=\{C_{154},C_{263},C_{326},C_{415},C_{541},C_{632}\}$
\\$S_{23}=\{C_{165},C_{256},C_{516},C_{562},C_{625},C_{651}\}$

$S_{24}=\{C_{166},C_{255},C_{525},C_{552},C_{616},C_{661}\}$
\\$S_{25}=\{C_{155},C_{266},C_{551},C_{515},C_{626},C_{662}\}$

$S_{26}=\{C_{156},C_{265},C_{526},C_{561},C_{615},C_{652}\}$
\\$S_{27}=\{C_{333},C_{444}\}$

$S_{28}=\{C_{334},C_{343},C_{344},C_{433},C_{434},C_{443}\}$
\\$S_{29}=\{C_{335},C_{353},C_{446},C_{464},C_{533},C_{644}\}$

$S_{30}=\{C_{336},C_{363},C_{445},C_{454},C_{544},C_{633}\}$
\\$S_{31}=\{C_{345},C_{364},C_{436},C_{453},C_{534},C_{643}\}$

$S_{32}=\{C_{346},C_{354},C_{435},C_{463},C_{543},C_{634}\}$
\\$S_{33}=\{C_{365},C_{456},C_{536},C_{564},C_{645},C_{653}\}$

$S_{34}=\{C_{366},C_{455},C_{545},C_{554},C_{663},C_{636}\}$
\\$S_{35}=\{C_{355},C_{466},C_{535},C_{553},C_{646},C_{664}\}$

$S_{36}=\{C_{356},C_{465},C_{546},C_{563},C_{635},C_{654}\}$
\\$S_{37}=\{C_{565},C_{556},C_{566},C_{655},C_{656},C_{665}\}$

$S_{38}=\{C_{555},C_{666}\}$
\end{multicols}

\newpage
\section{}\label{App:AppendixC}

\begin{multicols}{3}

 \begin{figure}[H]
\includegraphics[width=0.3\textwidth]{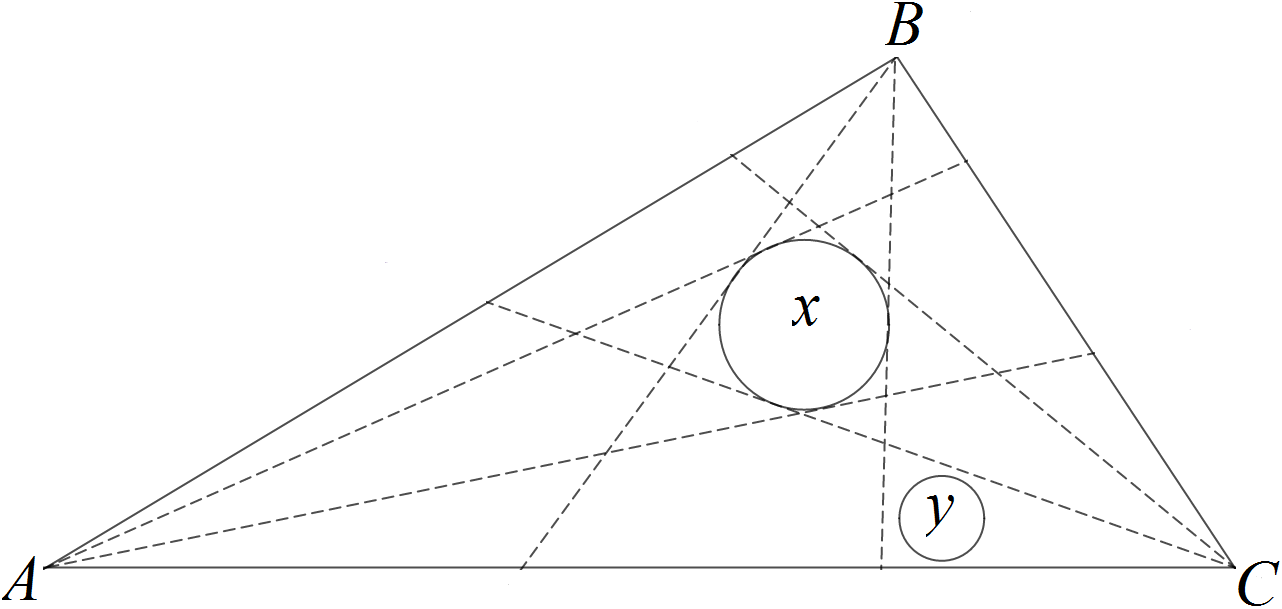}  
\caption{$S_{1}$}
  \label{fig:S1}
  \end{figure}
  
 \begin{figure}[H]
\includegraphics[width=0.3\textwidth]{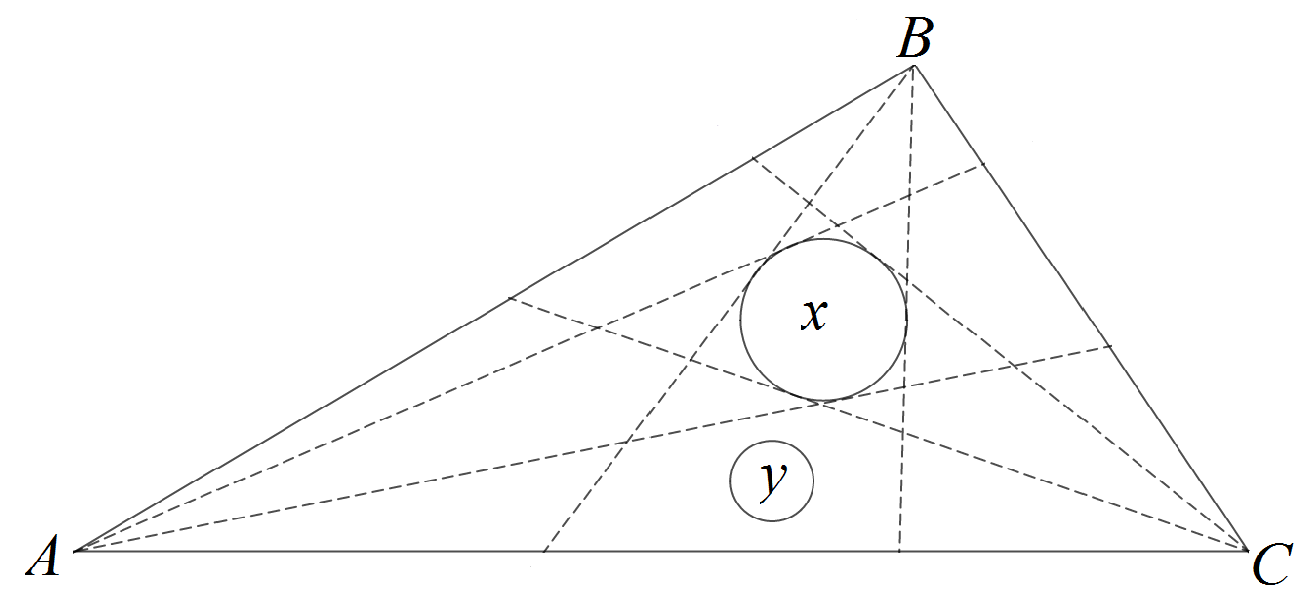}  
\caption{$S_{3}$}
  \label{fig:S3}
  \end{figure}

   \begin{figure}[H]
\includegraphics[width=0.3\textwidth]{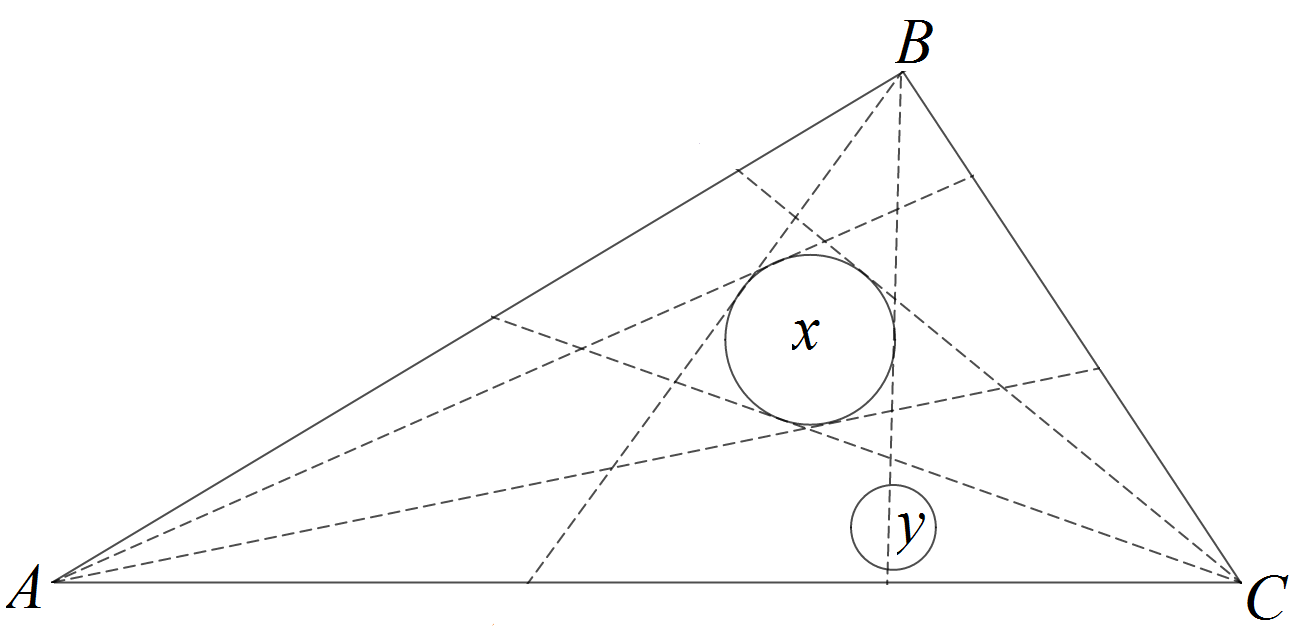}  
\caption{$S_{7}$}
  \label{fig:S7}
\end{figure}
 
 \begin{figure}[H]
\includegraphics[width=0.3\textwidth]{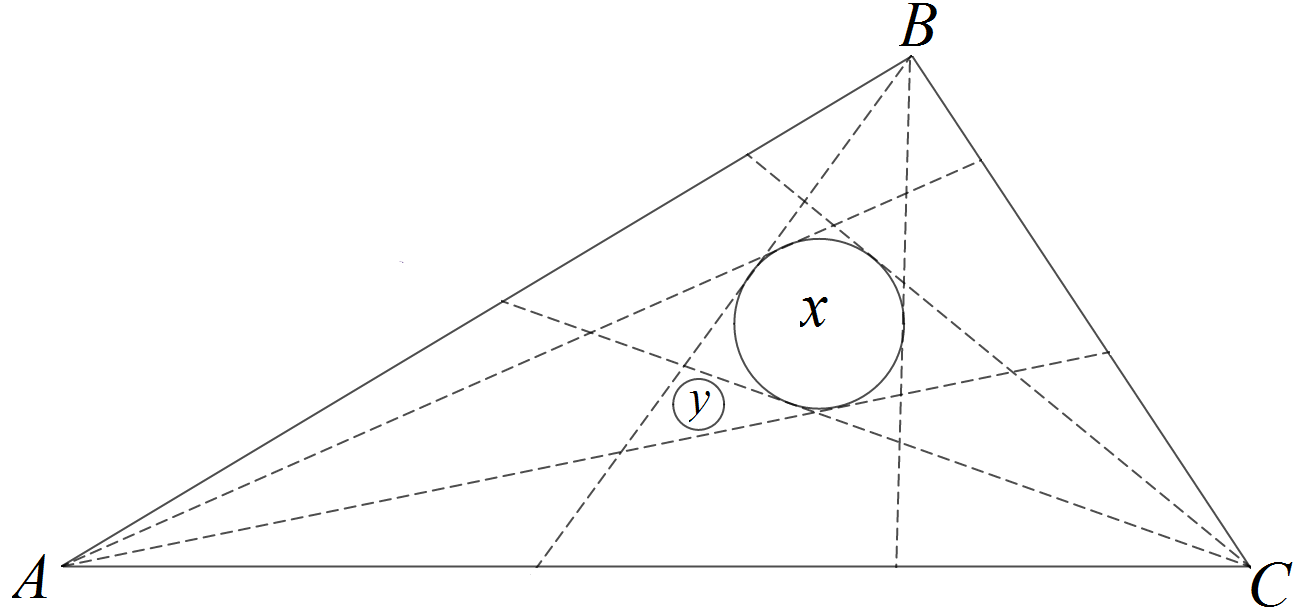}  
\caption{$S_{11}$}
  \label{fig:S11}
  \end{figure}
  
  \columnbreak
 
   \begin{figure}[H]
\includegraphics[width=0.3\textwidth]{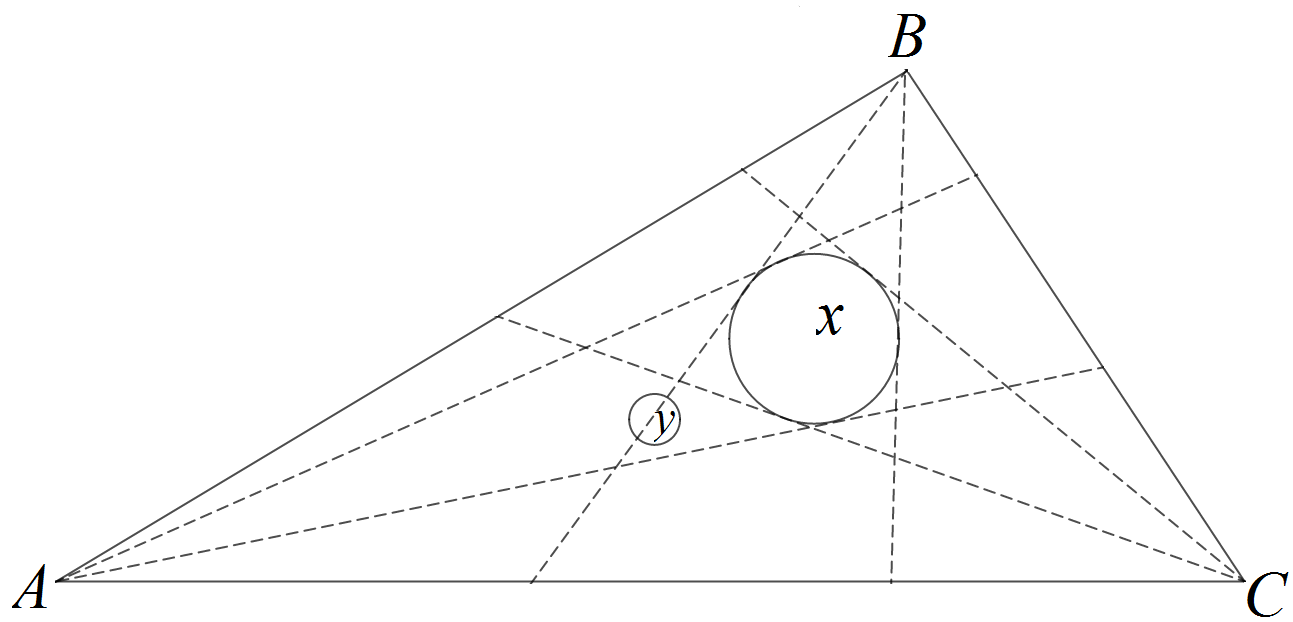}  
\caption{$S_{16}$}
  \label{fig:S16}
  \end{figure}
 
   \begin{figure}[H]
\includegraphics[width=0.3\textwidth]{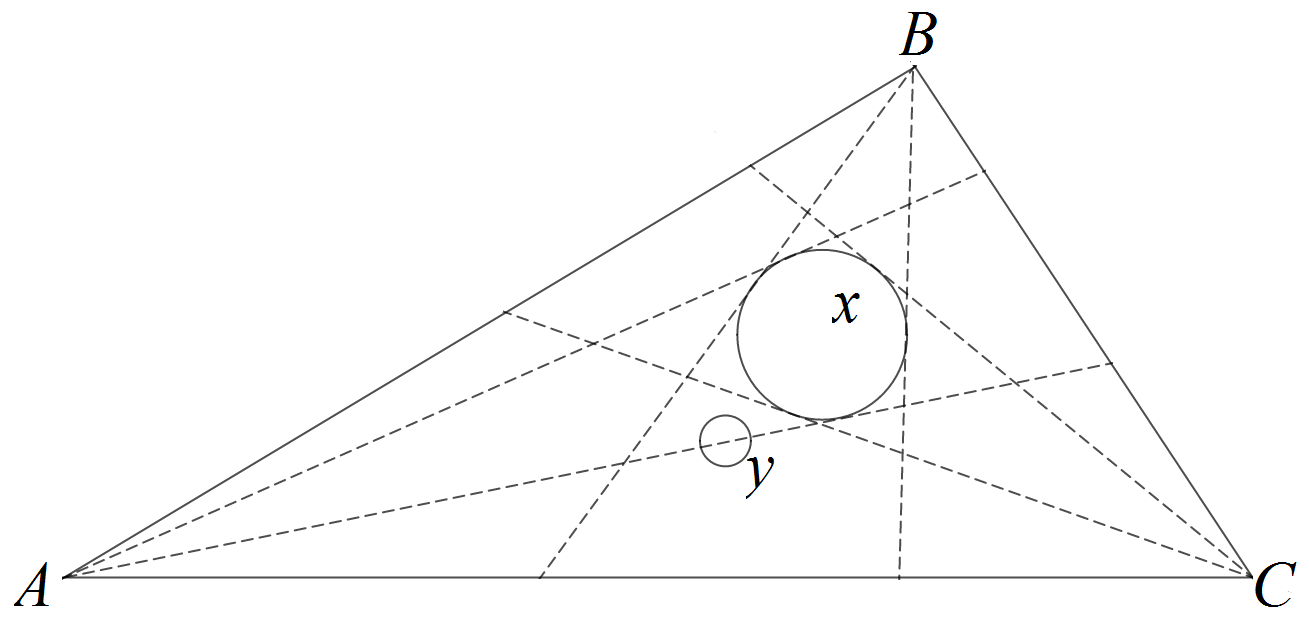}  
\caption{$S_{19}$}
  \label{fig:S19}
  \end{figure}
 
   \begin{figure}[H]
\includegraphics[width=0.3\textwidth]{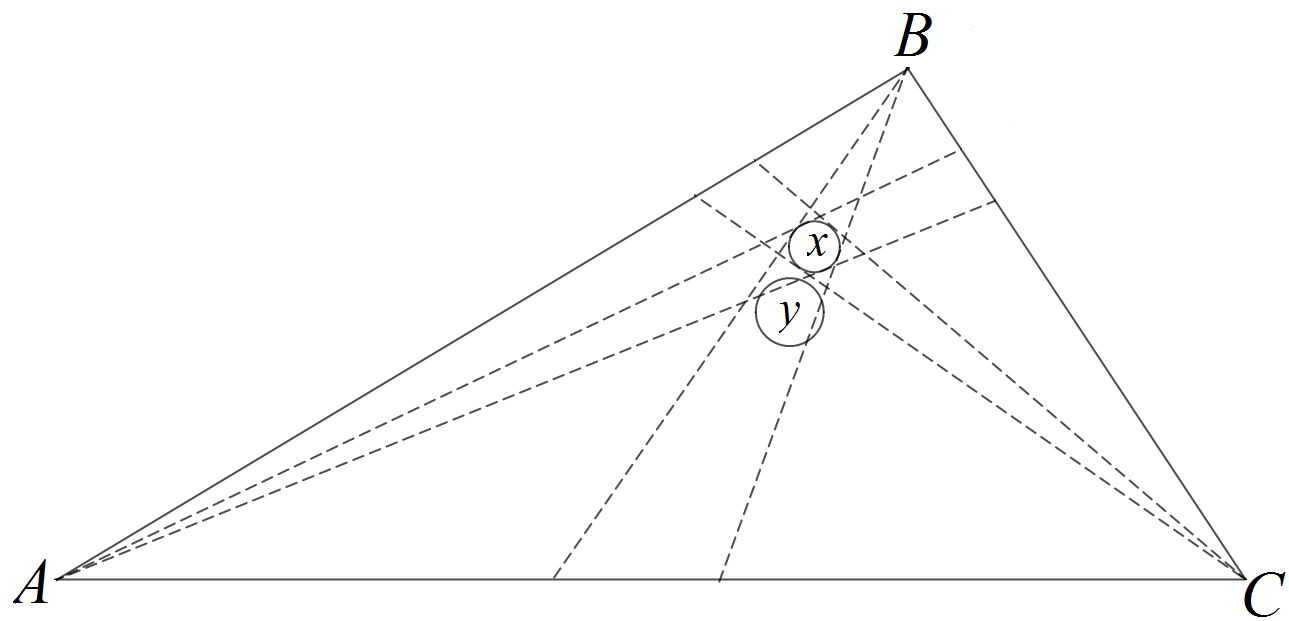}  
\caption{$S_{23}$}
  \label{fig:S23}
  \end{figure}
 
   \begin{figure}[H]
\includegraphics[width=0.3\textwidth]{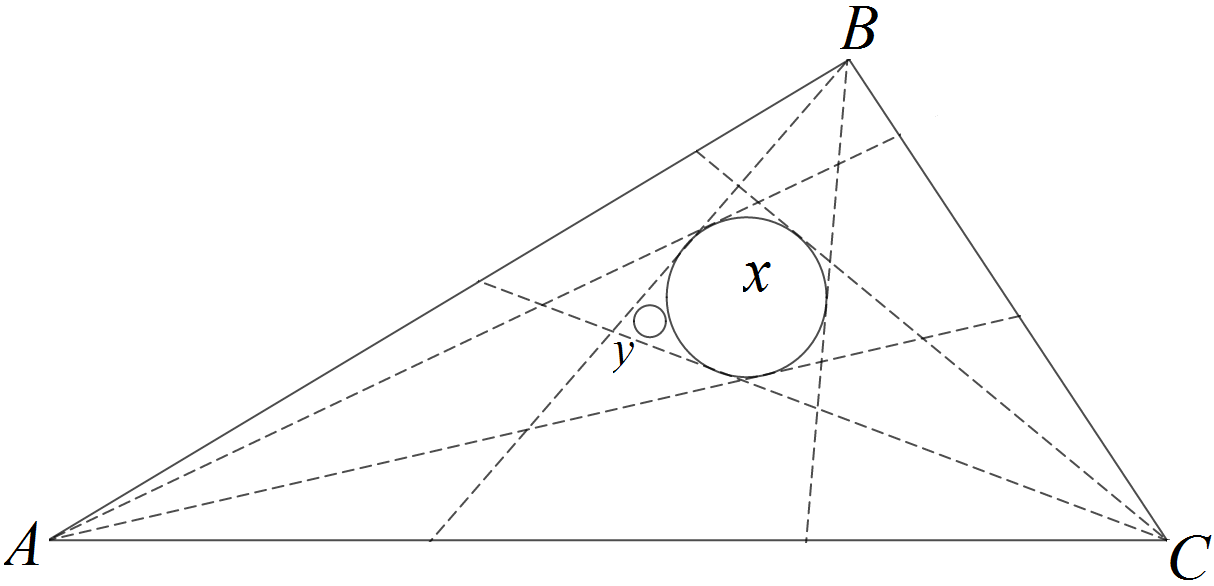}  
\caption{$S_{27}$}
  \label{fig:S27}
  \end{figure}
 
 \columnbreak
 
   \begin{figure}[H]
\includegraphics[width=0.3\textwidth]{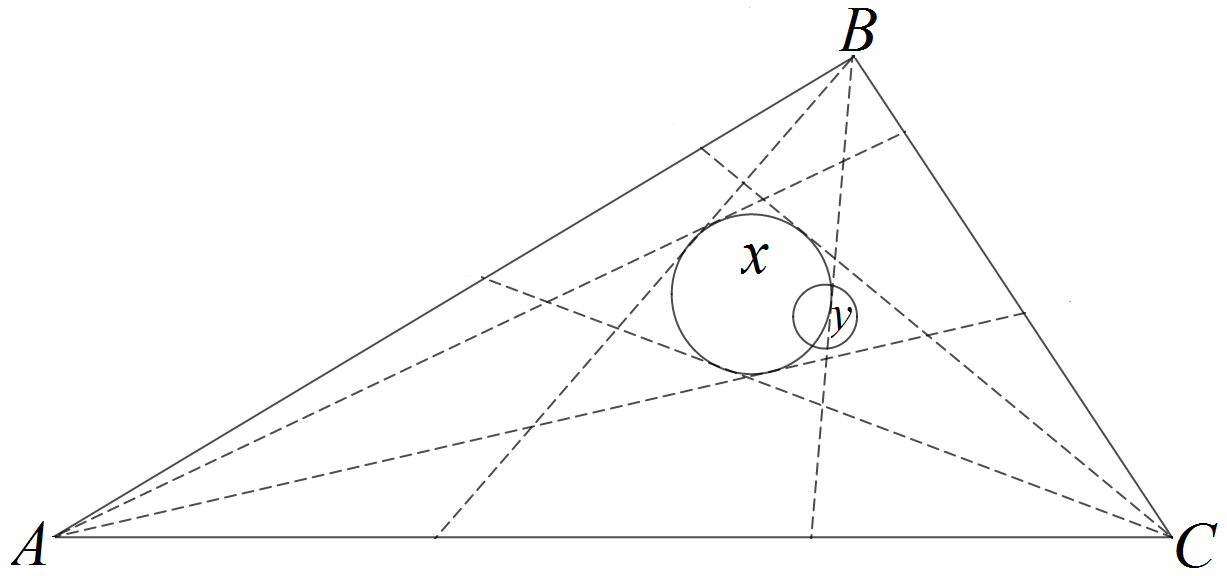}  
\caption{$S_{29}$}
  \label{fig:S29}
  \end{figure}
    \begin{figure}[H]
\includegraphics[width=0.3\textwidth]{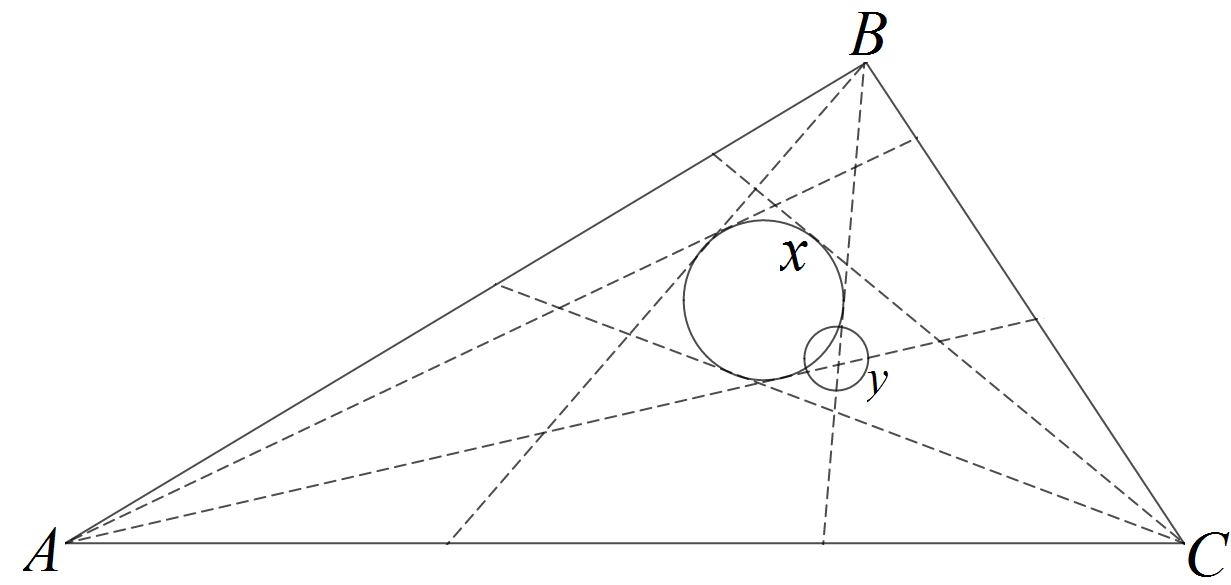}  
\caption{$S_{33}$}
  \label{fig:S33}
  \end{figure}
   \begin{figure}[H] 
\includegraphics[width=0.3\textwidth]{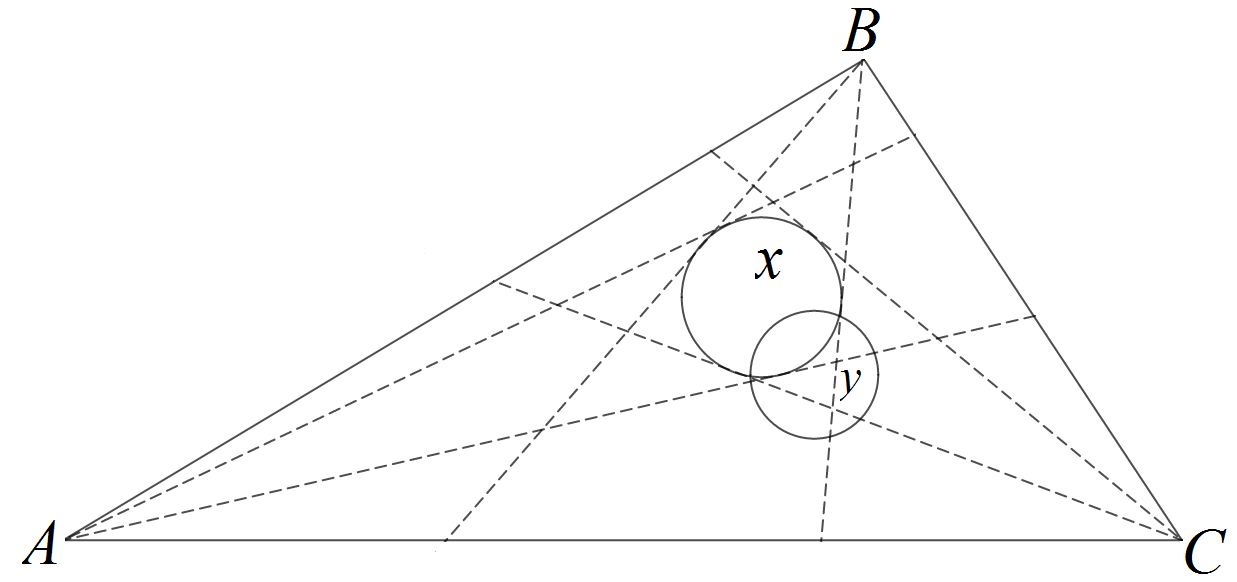}  
\caption{$S_{37}$}
  \label{fig:S37}
  \end{figure}

    
\end{multicols}

\begin{thebibliography}{9}
\bibitem{Adaricheva} 
K.~Adaricheva,
\textit{Representing Finite Convex Geometries by Relatively Convex Sets}.
Europ. J. of Combinatorics 37 (2014), 68-78. 

\bibitem{Adar2004} 
K.~Adaricheva,
\textit{Join-semidistributive Lattices of Relatively Convex Sets}. Contributions to General Algebra 14, Proceedings of the Olomouc Conference 2002 (AAA 64) and the Potsdam conference 2003 (AAA 65), Verlag Johannes Heyn, Klagenfurt, 2004, 1-14.

\bibitem{Adar2007} K.~Adaricheva, \emph{Realization of abstract convex geometries by point configurations. Part II}, preprint, 2007.

\bibitem{CzedliComm}
K.~Adaricheva, Private communication with G.~Cz\'edli, January 2013.

\bibitem{AdarGT} 
K.V.~Adaricheva, V.A.~Gorbunov and V.I.~Tumanov, 
\textit{Join Semidistributive Lattices and Convex Geometries}.
Advances in Mathematics 173 (2003), 1-49.

\bibitem{AN16} K.~Adaricheva and J.B.~Nation, \textit{ Bases of closure systems}, in Lattice Theory: Special Topics in Applications: Volume 2 (G.~Gr\"atzer and F.~Wehrung eds.), Birkhauser 2016, ISBN-13: 978-3319442358

\bibitem{AdarWild} 
K.~Adaricheva and M.~Wild, 
\textit{Realization of Abstract Convex Geometries by Point Configurations}, Europ. J. of Combinatorics 31 (2010), 379-400. 

\bibitem{Czedli} 
G.~Cz\'edli,
\textit{Finite Convex Geometries of Circles}.  
Discrete Mathematics 330 (2014), 61-75. 


\bibitem{D2}
R.P.~Dilworth,
\emph{Lattices with unique irreducible decompositions},
Annals of Math. (2) \textbf{41} (1940), 771--777.

\bibitem{EdelJam} 
P.H.~Edelman and R.E.~Jamison,
\textit{The Theory of Convex Geometries}.  
Geom Dedicata 19 (1985), 247-274. 

\bibitem{EdLa}
P.H.~Edelman and D.G.~Larman, \emph{On characterizing collections arising from
$N$-gons in the plane}, Geom. Dedicata 33(1990), 83--89.

\bibitem{FJN} R.~Freese, J.~Je\v{z}ek and J.B.~Nation, Free lattices, Mathematical Surveys and Monographs \textbf{42}, American Mathematical Society, Providence, RI, 1995.

 \bibitem{Kashetal} 
K. Kashiwabara, M. Nakamura and Y. Okamoto,
\textit{The Affine Representation Theorem for Abstract Convex Geometries}.  
Computational Geometry 30 (2005), 129-144.

\bibitem{Mn}
N.E.~Mn\"{e}v, \emph{The universality theorem on the classification problem of configuration varieties and convex polytopes varieties}, in Topology and Geometry--Rohlin Seminar (O.Ya.Viro, editor), v.1346 of Lecture Notes in Math., Springer-Verlag, Berlin, 1988, 527--544.

\bibitem{Mo85}
B.~Monjardet,
\emph{A use for frequently rediscovering a concept},
Order \textbf{1} (1985), 415--417.

\bibitem{RichRog} 
 M. Richter and L.G. Rogers,
\textit{Embedding Convex Geometries and a Bound on Convex Dimension}.
2015 http://arxiv.org/abs/1502.01941.

\bibitem{WS04} F.~Wehrung and M.~Semenova, \emph{Sublattices of lattices of convex subsets of vector spaces}, Algebra and Logic \textbf{43} (2004), 145--161.

\bibitem{Wi16} M.~Wild, \emph{Joy of Implications, a.k.a. pure Horn formulas: mainly a survey}, Theoretical Computer Science, to appear, arxiv.org/pdf/1411.6432.pdf
\end{thebibliography}
\end{document}